\documentclass[english,a4wide,12pt]{scrartcl}
\usepackage[T1]{fontenc}
\usepackage[latin9]{inputenc}
\usepackage{babel}
\usepackage{prettyref}
\usepackage{mathtools}
\usepackage{amsmath}
\usepackage{amsthm}
\usepackage{amssymb}
\usepackage[unicode=true,pdfusetitle,
 bookmarks=true,bookmarksnumbered=false,bookmarksopen=false,
 breaklinks=false,pdfborder={0 0 1},backref=false,colorlinks=false]
 {hyperref}

 \usepackage{color}

\makeatletter
\theoremstyle{plain}
\newtheorem{thm}{\protect\theoremname}[section]
\theoremstyle{definition}

\theoremstyle{plain}
\newtheorem{lem}[thm]{\protect\lemmaname}
\theoremstyle{remark}
\newtheorem{rem}[thm]{\protect\remarkname}
\theoremstyle{plain}
\newtheorem{prop}[thm]{\protect\propositionname}
\theoremstyle{plain}
\newtheorem{cor}[thm]{\protect\corollaryname}
\theoremstyle{definition}
\newtheorem{example}[thm]{\protect\examplename}
\theoremstyle{definition}

\@ifundefined{date}{}{\date{}}
\usepackage{prettyref}

\usepackage{enumerate}
\allowdisplaybreaks

\newcommand{\ind}{\operatorname{ind}}
\newcommand*{\e}{\mathrm{e}}
\renewcommand*{\i}{\mathrm{i}}

\newcommand{\R}{\mathbb{R}}
\newcommand{\T}{\mathbb{T}}
\newcommand{\N}{\mathbb{N}}
\newcommand{\CC}{\mathbb{C}}
\newcommand{\Z}{\mathbb{Z}}
\newcommand{\dom}{\operatorname{dom}}

\renewcommand{\d}{\,\mathrm{d}}
\renewcommand{\Re}{\operatorname{Re}}

\renewcommand{\tilde}{\widetilde}

\newcommand{\one}{\mathrm{I}}
\newcommand{\two}{\mathrm{II}}
\newcommand{\three}{\mathrm{III}}

\newrefformat{subsec}{Subsection \ref{#1}}
\newrefformat{prob}{Problem \ref{#1}}
\newrefformat{prop}{Proposition \ref{#1}}
\newrefformat{lem}{Lemma \ref{#1}}
\newrefformat{thm}{Theorem \ref{#1}}
\newrefformat{cor}{Corollary \ref{#1}}
\newrefformat{rem}{Remark \ref{#1}}
\newrefformat{exa}{Example \ref{#1}}
\newrefformat{sub}{Subsection \ref{#1}}
\newrefformat{eq}{(\ref{#1})}

\theoremstyle{definition}

\newrefformat{hyp}{Hypotheses \ref{#1}}

\makeatother

\providecommand{\corollaryname}{Corollary}
\providecommand{\definitionname}{Definition}
\providecommand{\examplename}{Example}
\providecommand{\lemmaname}{Lemma}
\providecommand{\problemname}{Problem}
\providecommand{\propositionname}{Proposition}
\providecommand{\remarkname}{Remark}
\providecommand{\theoremname}{Theorem}

\begin{document}
\title{Asymptotic Stability of port-Hamiltonian Systems}
\author{Marcus Waurick\thanks{Institut f\"ur Angewandte Analysis, TU Bergakademie Freiberg, Germany}\hspace{0.1cm} and
Hans Zwart\thanks{Department of Applied Mathematics, University of Twente, Twente, The Netherlands and  Department of Mechanical Engineering, Technische Universiteit Eindhoven, 5600 MB Eindhoven, The Netherlands}}
\maketitle
\begin{abstract}
We characterise asymptotic stability of port-Hamiltonian systems by means of matrix conditions using well-known resolvent criteria from $C_0$-semigroup theory. The idea of proof is based on a recent characterisation of exponential stability established in \cite{TW22_expstab}, which was inspired by a structural observation concerning port-Hamiltonian systems from \cite{PTWW_22}. We apply the result to study the asymptotic stability of a network of vibrating strings.
\end{abstract}
\textbf{Keywords} port-Hamiltonian systems, stability,
$C_{0}$-semigroup, Infinite-dimensional systems theory\\
\label{=00005Cnoindent}\textbf{MSC2020 }93D23, 37K40, 47D06, 34G10

\tableofcontents{}

\section{Introduction}

In this note we discuss the stability of port-Hamiltonian partial differential equations (p.d.e.'s) of the form
\begin{equation}
\label{eq:1HZ}
   \frac{\partial x}{\partial t}(\zeta,t) =  \sum_{k=0}^N P_k \frac{\partial^k }{\partial \zeta^k} \left( \mathcal{H}(\zeta) x(\zeta,t)\right)
\end{equation}
on the spatial interval $[a,b]$. Here, $P_k \in \CC^{n\times n}$ satisfy $P_k^*= (-1)^{k+1} P_k$ ($k\in\{0,\cdots, N\}$) and $P_N$ is invertible. The \emph{Hamiltonian density} $\mathcal{H}\colon (a,b)\to \CC^{n\times n}$ is uniformly bounded, ${\mathcal H}(\zeta)^* = {\mathcal H}(\zeta)$ and there exists an $m>0$ such that ${\mathcal H}(\zeta) \geq mI$ for almost all $\zeta \in [a,b]$. The above p.d.e.\ is completed with boundary conditions given as
\begin{equation}
\label{eq:2HZ}
  W_B \left[ \begin{array}{c}  ({\mathcal H}x)(b,t)\\ \vdots\\ \frac{\partial^{N-1}(\mathcal{H} x)}{\partial \zeta^{N-1}}(b,t)\\({\mathcal H}x)(a,t)\\ \vdots\\  \frac{\partial^{N-1} (\mathcal{H} x)}{\partial \zeta^{N-1}}(a,t)\end{array} \right] = 0,
\end{equation}
where $W_B$ is an $nN \times 2nN$ matrix which we assume to be of full rank.

One possible approach to address properties of the above system is to invoke the theory of strongly continuous semigroups of bounded linear operators, in that one reformulates the above as an abstract ordinary differential equation in an (infinite-dimensional) state space $X$. For this, a suitable functional analytic setting is conveniently formulated in the weighted Hilbert space $X=L_{2,\mathcal{H}}((a,b);{\mathbb C}^n)$, which coincides with $L_2((a,b);{\mathbb C}^n)$ as set and is endowed with the weighted inner product given by
\begin{equation}
\label{eq:3HZ}
   \langle f,g\rangle_{\mathcal{H}}
   =\frac{1}{2} \int_a^b g(\zeta)^* {\mathcal H}(\zeta) f(\zeta) d\zeta.
\end{equation}
$X$ is known as the {\em energy space}. In \cite{GoZM05} necessary and sufficient conditions are given when the operator associated with  the above p.d.e.\ plus boundary conditions generates a contraction semigroup. Since the notation will be used in our stability result, we repeat this theorem. We define the $nN\times nN$ matrix $Q$ and the $2nN \times 2nN$ matrix $R_{ext}$ as
\begin{equation}
\label{eq:Q}
  Q= [Q_{ij}], i,j \in \{1,\cdots, N\} \mbox{ with } Q_{ij} = \begin{cases} 0 & i+j >N \\ (-1)^{i-1} P_k & i+j-1 =k
  \end{cases}
\end{equation}
and
\[
  R_{ext} = \frac{1}{\sqrt{2}} \begin{bmatrix} Q & - Q \\ I_{nN} & I_{nN} \end{bmatrix}.
\]Note that $R_{ext}$ is invertible since $P_{N}$ is invertible.
\begin{thm}[\cite{GoZM05,JMZ_2015,JZ}]\label{thm:gencon} 
Consider the operator 
\begin{equation} 
\label{eq:4HZ}
  A  = \left( P_N\frac{d^N}{d\zeta^N} + \cdots + P_1\frac{d}{d \zeta} + P_0\right) \mathcal{H}
\end{equation}
with domain
 \begin{equation}
 \label{eq:5HZ}
    \dom(A)= \{ x\in L_{2,\mathcal{H}}((a,b); {\mathbb C}^n) \mid  \mathcal{H}x \in H^N((a,b);{\mathbb C}^n), W_B \begin{bmatrix} (\mathcal{H}x)(b) \\ \vdots \\ \frac{d^{N-1}\mathcal{H}x}{d\zeta^{N-1}}(b) \\ (\mathcal{H}x)(a)\\ \vdots \\ \frac{d^{N-1}\mathcal{H}x}{d\zeta^{N-1}}(a)  \end{bmatrix} = 0\},
 \end{equation}
where $W_B$ is a full rank $nN\times 2nN$ matrix.
Then the following conditions are equivalent:
\begin{enumerate}[(i)]
 \item 
   $A$ generates a contraction semigroup $X$;
 \item 
   $W_B$ satisfies
\begin{equation}
    \label{eq:WB}
   W_{B} R_{ext}^{-1} \begin{bmatrix}
0 & I_{nN}\\
I_{nN} & 0
\end{bmatrix}\left(W_{B} R_{ext}^{-1}\right)^{\ast}\geq 0;
\end{equation}
\item
  For all $x \in \dom(A)$ there holds
  \[
    \langle Ax, x\rangle_{\mathcal H} + \langle x, Ax\rangle_{\mathcal H} \leq 0.
  \]
\end{enumerate} 
\end{thm}
Although the characterisation of contraction semigroups is given for the operator associated to the  p.d.e.\ (\ref{eq:1HZ}) with boundary conditions (\ref{eq:2HZ}), other properties have only been studied for the case $N=1$, see e.g. \cite{JZ}. Recently, for that class the characterisation of exponential stability was found.
In order to state this result, we need to consider the following parametrised ordinary differential equation
\begin{equation}
\label{eq:8HZ}
    \frac{dv}{d\zeta}(\zeta) = \i \omega P_1^{-1}(\mathcal{H}(\zeta)^{-1}+P_0) v(\zeta)\text{ for }\zeta\in (a,b) \mbox{ and } \omega \in {\mathbb R}.
\end{equation}
Let $\Phi_{\omega}$ denote its fundamental solution\footnote{In Appendix \ref{app:wp} we show that this exists and is uniquely determined, even when the right-hand side of (\ref{eq:8HZ}) is not continuous.}, i.e., 
\begin{equation}
\label{eq:Phit}
  \Phi_{\omega} \colon [a,b)\to \CC^{n\times n} \mbox{ with }  \Phi_{\omega}(a) = I_n,
\end{equation} 
and satisfies (\ref{eq:8HZ}). 
The characterisation of exponential stability now reads as follows.
\begin{thm}[\cite{TW22_expstab}]
\label{thm:charexpo}  
Let $(A, \dom(A))$ be given as in Theorem \ref{thm:gencon} for $N=1$, and be such that $A$ generates a contraction semigroup on $X$. Furthermore, assume that 
\begin{equation}
\label{eq:(B)}
   \sup_{\omega\in\mathbb{R}}\|\Phi_{\omega}\|_{\infty}=\sup_{\omega \in\mathbb{R}}\sup_{\zeta\in(a,b)}\|\Phi_{\omega}(\zeta)\|<\infty.
\end{equation}
Then the following conditions are equivalent:
\begin{enumerate}[(i)]
\item $\left(T(t)\right)_{t \geq 0}$ is exponentially stable; 
\item  $Q_{\omega} \coloneqq W_B\begin{bmatrix} \Phi_{\omega}(b) \\ I_n \end{bmatrix}$ is invertible and $\sup_{\omega \in \R}\| Q_{\omega}^{-1}\|<\infty$.
\end{enumerate}
\end{thm} 

Inspired by this theorem, we characterise the semi-uniform/asymptotic stability of the contraction semigroup. For this, we invoke the celebrated Batty--Duyckaerts theorem, \cite{Batty_Duyckaerts_2008}. This result can be formulated for the general class, and so we do not assume that $N=1$.  We introduce the following extension of the differential equation, see also (\ref{eq:8HZ}),
\begin{equation}
\label{eq:12HZ}
    \frac{dv}{d\zeta}(\zeta) = {\mathcal P}_{\lambda}(\zeta)v(\zeta)\text{ for }\zeta\in (a,b) \mbox{ and } \lambda \in {\mathbb C},
\end{equation}
where $v$ takes its values in ${\mathbb C}^{nN}$  and 
\begin{equation}
\label{eq:12aHZ}
  {\mathcal P}_{\lambda}(\zeta) = \begin{bmatrix} 0& I & 0 & \cdots &  0 \\
  0& 0 & I & 0& \vdots\\
  \vdots & &\ddots & \ddots & 0\\
  0 & \cdots &\cdots & 0 & I\\
  \lambda P_N^{-1} {\mathcal H}(\zeta)^{-1}-P_N^{-1} P_0 & -P_N^{-1} P_1 & \cdots & \cdots &-P_N^{-1}P_{N-1}
  \end{bmatrix}
\end{equation}
Let $\Psi_{\lambda}$ denote its fundamental solution\footnote{In Appendix \ref{app:wp} we show that this exist, even when the right-hand side of (\ref{eq:12HZ}) is not continuous}, i.e., 
\begin{equation}
\label{eq:13HZ}
  \Psi_{\lambda} \colon [a,b)\to \CC^{nN\times nN} \mbox{ with }  \Psi_{\lambda}(a) = I_{nN},
\end{equation} 
and satisfies (\ref{eq:12HZ}). Note that for $N=1$ we have that $\Psi_{i\omega} = \Phi_{\omega}$, see (\ref{eq:8HZ}), (\ref{eq:Phit}).
\begin{thm}\label{thm:main1}
Assume  that $(A, \dom(A))$ as given in (\ref{eq:4HZ}) and (\ref{eq:5HZ}) generates a contraction semigroup $\left(T(t)\right)_{t\geq 0}$ on $X$. 
 Then the following conditions are equivalent:
\begin{enumerate}[(i)]
\item 
  $\left(T(t)\right)_{t\geq 0}$ is \emph{semi-uniformly stable (with decay rate $f\colon [0,\infty)\to [0,\infty)$)}, i.e., $f(t)\to 0$ as $t\to\infty$ and for all $x\in \dom(A)$
\[
   \|T(t)x\|_{\mathcal H}\leq f(t)\|x\|_{\dom(A)} \quad t\geq 0.
\]
\item
  $\left(T(t)\right)_{t\geq 0}$ is {\em asymptotically stable}, i.e, for all $x \in X$ there holds $T(t) x \rightarrow 0$ as $t \rightarrow \infty$;
\item
  The operator $A$ with domain $\dom(A)$ has no eigenvalues on the imaginary axis;
\item 
  The square matrix ${\mathcal Q}(i\omega) \coloneqq W_B\begin{bmatrix} \Psi_{i\omega} (b) \\ I_{nN}\end{bmatrix}$ is invertible for all $\omega \in \R$.
\item For all $\omega \in {\mathbb R}$ the set \[\{ v_a \in {\mathbb C}^{nN} \mid{\mathcal Q}(i\omega)v_a=0 \mbox{ and } v_a^* \left[\begin{smallmatrix} \Psi_{i\omega} (b)^* & I_{nN} \end{smallmatrix}\right]  \left[\begin{smallmatrix} Q & 0 \\ 0 & -Q \end{smallmatrix} \right] \left[\begin{smallmatrix} \Psi_{i\omega} (b) \\ I_{nN} \end{smallmatrix}\right] v_a =0\}\]contains only the zero element. 
\end{enumerate}
\end{thm}

The proof of this theorem will be given in the following section. Furthermore, we also adopt the insights drawn from this equivalence to shed some more light on the characterisation of exponential stability from \cite{TW22_expstab}. We exemplify our results in the sections \ref{Sec:3} and \ref{Sec:4}. We provide a conclusion afterwards. This manuscript is supplemented with an appendix gathering known material for non-autonomous o.d.e.'s with bounded and measurable coefficients.

\section{Characterisation of asymptotic/semi-uniform stability}

We recall the celebrated Batty--Duyckaerts result on semi-uniform stability.
\begin{thm}[{{Batty--Duyckaerts, \cite{Batty_Duyckaerts_2008}; see, \cite[Theorem 3.4]{CST20}}}]\label{thm:BD} Let $\left(T(t)\right)_{t\geq 0}$ be a boun\-ded $C_0$-semigroup on a Banach space $X$, with generator $\mathcal{A}$, $\mu\in \rho(\mathcal{A})$. Then the following conditions are equivalent:
\begin{enumerate}[(i)]
\item $\sigma(\mathcal{A})\cap \i\R=\emptyset$;
\item $\lim_{t\to\infty} \|T(t)(\mu-\mathcal{A})^{-1}\|=0$;
\item there exists a function $f\colon [0,\infty)\to[0,\infty)$ with $f(t)\to 0$ as $t\to\infty$ and such that
\[
  \forall x_0\in \dom(\mathcal{A})\colon \|T(t)x\|\leq f(t)\|x_0\|_{\dom(\mathcal{A})},
\]i.e., $T$ is \emph{semi-uniformly stable (with decay rate $f$)}.
\end{enumerate}
\end{thm}

In order to apply the Batty--Duyckaerts result, we provide some (standard) observations.
\begin{thm}\label{thm:Acomp} The operator $A$ with domain $\dom(A)$ as given by (\ref{eq:4HZ}) and (\ref{eq:5HZ}) has compact resolvent on the energy space $X$. 
\end{thm}
\begin{proof}
The assertion for $\mathcal{H}=I_n$ follows upon using the Arzel\`a--Ascoli Theorem since $\dom(A)$ embeds continuously into $H^N((a,b);{\mathbb C}^n)$. The general case then follows since $X$ is isomorphic to $L^2((a,b);{\mathbb C}^n)$. 
\end{proof}
\begin{lem}\label{lem:spectrumcomp} Let $\mathcal{B}$ be a closed, densely defined, linear operator on a Hilbert space $X$, with $\rho(\mathcal{B})\neq \emptyset$. If $\dom(\mathcal{B})\hookrightarrow X$ is compact, then 
\[\sigma(\mathcal{B})=\{\lambda \in \CC \mid 0<\dim(\ker(\lambda I-\mathcal{B}))<\infty \}\eqqcolon \sigma_p(\mathcal{B}).\]
\end{lem}
\begin{proof}
For the proof we invoke standard theory of (unbounded) Fredholm operators, see e.g., \cite[Theorem 3.6]{GW16} for an overview and suitable references. Let $\mu\in \rho(\mathcal{B})$ and $\lambda\in \sigma(\mathcal{B})$. Since $\mu I -\mathcal{B}$ is continuously invertible, $\mu I -\mathcal{B}$ is a Fredholm operator. The compactness of $\dom(\mathcal{B})\hookrightarrow X$ yields that $(\mu I-\mathcal{B})^{-1}$ is compact and, hence, $(\lambda-\mu)I$ is a relatively compact perturbation of $\mu I-\mathcal{B}$ and so $\lambda I-\mu I+(\mu I-\mathcal{B})=\lambda I-\mathcal{B}$ is Fredholm and $\ind(\mu I-\mathcal{B})=\ind(\lambda I-\mathcal{B})$. In particular, $\dim(\ker(\lambda I-\mathcal{B}))<\infty $. Since $\lambda \in \sigma(\mathcal{B})$, $\lambda I-\mathcal{B}$ fails to be continuously invertible, which can, thus, only happen if $\ker(\lambda I-\mathcal{B})\neq \{0\}$; whence $\dim(\ker(\lambda I-\mathcal{B}))>0$.
\end{proof}

The application of Theorem \ref{thm:BD} to port-Hamiltonian systems, reads as follows.
\begin{thm}\label{thm:susA} Assume $(A,\dom(A))$ as in Theorem \ref{thm:gencon} generates a contraction semigroup $\left( T(t) \right)_{t\geq 0}$. Then the following conditions are equivalent:
\begin{enumerate}
  \item[(i)] $\left(T(t)\right)_{t\geq 0}$ is semi-uniformly stable;
  \item[(ii)] $\sigma_p(A)\cap \i\R=\emptyset$. 
\end{enumerate}
\end{thm}
\begin{proof}
The result is a straightforward consequence of Theorem \ref{thm:BD} in conjunction with the observations in Theorem \ref{thm:Acomp} and Lemma \ref{lem:spectrumcomp} to determine the type of spectrum $A$ can possibly have.
\end{proof}

\mbox{}From this result we see that the eigenvalues of $A$ will play an essential role. Therefore, we characterise these in the following lemma.
\begin{lem}
\label{L:eig} Let $(A,\dom(A))$ be given by (\ref{eq:4HZ}) and (\ref{eq:5HZ}) then $\lambda$ is an eigenvalue if and only if ${\mathcal Q}(\lambda) \coloneqq W_B\begin{bmatrix} \Psi_{\lambda} (b) \\ I_{nN}\end{bmatrix}$ is not invertible, see (\ref{eq:13HZ}).
\end{lem}
\begin{proof}
Let $\lambda \in {\mathbb C}$ be an eigenvalue of $(A,\dom(A))$, i.e., $A x= \lambda x$ for $0\neq x \in \dom(A)$. Using (\ref{eq:4HZ}), we have that 
  \[
    \lambda x = Ax \Rightarrow \lambda x(\zeta) = \left(P_N \frac{d^N}{d\zeta^N} + \cdots + P_1 \frac{d}{d\zeta} + P_0\right)({\mathcal H} x)(\zeta)
  \]
  Introducing $v_1(\zeta) := ({\mathcal H} x)(\zeta)$, we can rewrite the above o.d.e.\ as
  \begin{equation}
  \label{eq:14HZ}
    \left(P_N \frac{d^N}{d\zeta^N} + \cdots + P_1 \frac{d}{d\zeta} + P_0\right)v_1(\zeta) - \lambda {\mathcal H}^{-1}(\zeta) v_1(\zeta) =0.
  \end{equation}
  Next we introduce
  \[
    v(\zeta) = \begin{bmatrix} v_1(\zeta) \\ \vdots\\ \frac{d^{N-1}v_1}{d\zeta^{N-1}}(\zeta) \end{bmatrix}
    \quad  
   \]
  and so with ${\mathcal P}_{\lambda}$ as defined in (\ref{eq:12aHZ}) we can formulate the $N$-th order o.d.e.\ (\ref{eq:14HZ}) as the first order o.d.e.\
  \[
    \frac{dv}{d\zeta}(\zeta) = {\mathcal P}_{\lambda}(\zeta) v(\zeta)
  \]
  and thus $v(\zeta) = \Psi_{\lambda}(\zeta) v(a)$. 
  
  Since $x$ was an eigenvector it is an element of $\dom(A)$ and thus $v_1={\mathcal H}x$ satisfies the boundary conditions as given in (\ref{eq:5HZ}). We can equivalently formulate these conditions by using $v$. This gives
  \[
    W_B \begin{bmatrix} v(b) \\ v(a) \end{bmatrix} =0 \Leftrightarrow W_B \begin{bmatrix} \Psi_{\lambda}(b) \\ I_{nN}  \end{bmatrix} v(a) =0.
  \]
  If $v(a)$ would be zero, then the function $v$ would be zero, and so $x$. This is not possible since $x$ is an eigenvector. So for the above to hold, the square matrix $W_B \begin{bmatrix} \Psi_{\lambda}(b) \\ I_{nN}  \end{bmatrix}$ must be singular. 
  
  Similarly, when $W_B \begin{bmatrix} \Psi_{\lambda}(b) \\ I_{nN}  \end{bmatrix}$ is singular, then we choose $v(a)$ in its kernel, and we define $v(\zeta) :=\Psi_{\lambda}(\zeta) v(a)$. It is straightforward to see that $x := {\mathcal H}^{-1} v_1$ is in the domain of $A$ and satisfies $Ax=\lambda x$.
\end{proof}

\subsection*{Proof of Theorem \ref{thm:main1}}

In this section we prove Theorem \ref{thm:main1}. 
\begin{itemize}
\item {\textbf{(i)}\ $\Rightarrow$ \textbf{(ii)}} Let $x_0 \in X$, and choose $\varepsilon >0$. Since the domain of $A$ lies dense in $X$ we can find an $x_1 \in \dom(A)$ such that $\|x_1 - x_0\| \leq \varepsilon$. Next choose  $t_1>0$ such that $\|T(t_1)x_1\| \leq \varepsilon$. Now
\[
  \|T(t_1) x_0\| \leq  \|T(t_1) (x_0-x_1) \| + \|T(t_1)x_1 \| \leq \|x_0-x_1\| + \|T(t_1)x_1 \| \leq 2 \varepsilon,
\]
where we have used that $\left(T(t)\right)_{t\geq 0}$ is a contraction semigroup. Using this once more, we find for $t \geq t_1$
\[
  \|T(t) x_0\| = \|T(t-t_1) T(t_1)x_0\| \leq  \| T(t_1)x_0\| \leq 2\varepsilon
\]
which shows that the semigroup is asymptotically stable.
\item  {\textbf{(ii)}\ $\Rightarrow$ \textbf{(iii)}} If $\lambda$ is an eigenvalue on the imaginary axis, then for the corresponding eigenvector $x_0$ there holds $T(t) x_0 = e^{\lambda t} x_0$. This will not converge to zero as $t \rightarrow \infty$. Hence since not (iii)\ implies not (ii), we have shown the implication form (ii)\ to (iii).
\item
 {\textbf{(iii)}\ $\Rightarrow$ \textbf{(i)}} See Theorem \ref{thm:susA}.
\item 
  {\textbf{(iii)}\ $\Leftrightarrow$ \textbf{(iv)}} Follows from Lemma \ref{L:eig}.
\item{\textbf{(iv)}\ $\Leftrightarrow$ \textbf{(v)}} It is clear that part (iv) implies part (v). So we concentrate on the other direction. We do this by showing that  if $v_a \in {\mathbb C}^{nN}$ is such that ${\mathcal Q}(i\omega)v_a =0$, then it also satisfies the second equality in item (v).  

Let $0\neq v_a \in {\mathbb C}^{nN}$ be such that ${\mathcal Q}(i\omega)v_a =0$. We define $v(\zeta)$ as the solution of (\ref{eq:12HZ}) with initial condition $v(a) = v_a$. Then it is easy to see that $v$ can be written as
\[
  v(\zeta) = \begin{bmatrix} v_1(\zeta) \\ \vdots\\ \frac{d^{N-1}v_1}{d\zeta^{N-1}}(\zeta) \end{bmatrix}
\]
with $v_1$ satisfying (\ref{eq:14HZ}) for $\lambda = i\omega$. Since ${\mathcal H}$ is symmetric, we find that
\begin{align*}
  0 =&\ v_1(\zeta)^* \left(i\omega {\mathcal H}(\zeta)^{-1} v_1(\zeta) \right) + \left(i\omega {\mathcal H}(\zeta)^{-1} v_1(\zeta) \right)^* v_1(\zeta) \\
  =&\ v_1(\zeta)^* \left( \left(P_N \frac{d^N}{d\zeta^N} + \cdots + P_1 \frac{d}{d\zeta} + P_0\right)v_1(\zeta)\right) + \\
  &\ \left( \left(P_N \frac{d^N}{d\zeta^N} + \cdots + P_1 \frac{d}{d\zeta} + P_0\right)v_1(\zeta)\right) ^* v_1(\zeta).
\end{align*}
Integrating this expression from $\zeta=a$ till $b$ and using  \cite[Lemma 3.1]{GoZM05}, we find that
\[
  0= v(b)^* Q v(b) - v(a)^*Qv(a).
\]
Since $v(a) =v_a$ and $v(b) = \Psi_{i\omega}(b)v_a$, this is the second equality in part (v).
\end{itemize}


\section{Exponential stability revisited}\label{sec:exprev}

Before we turn to our example we shortly revisit exponential stability and reformulate the conditions from \cite{TW22_expstab}. We start off with an elementary observation on families of invertible matrices.

\begin{prop}\label{prop:unifinv} Let $Q\colon \R\to \CC^{n\times n}$ be bounded and continuous. Assume that for all $t\in \R$, $Q(t)$ is invertible. Then the following conditions are equivalent:
\begin{enumerate}
  \item[(i)] $\sup_{t\in \R} \|Q(t)^{-1}\|<\infty$;
  \item[(ii)] for all $v\in \CC^n$, $v\neq 0$, $\liminf_{t\to\pm \infty} \|Q(t)v\|>0$.
\end{enumerate}
\end{prop}
\begin{proof}
 (i)$\Rightarrow$(ii): Let $M\coloneqq\sup_{t\in \R} \|Q(t)^{-1}\|$ and $v\in \CC^n$, $v\neq 0$. Then for all $t\in \R$
 \[
     \|v\|=\|Q(t)^{-1}Q(t)v\|\leq M\|Q(t)v\|.
 \]
 Taking the $\liminf$ on both sides of the inequality yields the assertion.
 
 (ii)$\Rightarrow$(i): Consider 
 \[
    B\coloneqq \overline{\{Q(t); t\in \R\}}\subseteq \CC^{n\times n}.
 \]
 By the boundedness of $Q$, $B$ is bounded and closed; hence compact. Next, we show that for all $Q_0\in B$, $|\det Q_0|>0$. For this let $Q_0\in B$. Then there exists a sequence  $(t_k)_{k \in {\mathbb N}}$ such that $Q(t_k)\to Q_0$ as $k\to \infty$. If $(t_k)_{k \in {\mathbb N}}$ accumulates at some $s\in \R$, we find a subsequence (not relabelled) such that $\lim_{k\to\infty} t_k=s$. By continuity, we infer $Q_0=Q(s)$, which is invertible by assumption. Hence, $\det Q_0\neq 0$. If $(t_k)_{k \in {\mathbb N}}$ accumulates at $-\infty$ or $\infty$, we may assume without restriction (the other case is similar), that we find a subsequence (again not relabelled) such that $\lim_{k\to\infty} t_k=\infty$. Let $v\in \CC^n$, $v\neq 0$. We infer, using (ii),
 \[
    \|Q_0 v\| =\lim_{k\to\infty} \|Q(t_k)v\|\geq \liminf_{t\to\pm \infty} \|Q(t)v\|>0.
 \]
 Hence, $0$ cannot be an eigenvalue of $Q_0$, that is, $\det Q_0\neq 0$. By compactness of $B$ and continuity of $\det$, we thus obtain that $|\det R|\geq \varepsilon>0$ for all $R\in B$ and some $\varepsilon>0$. Using the boundedness of $Q$ and Cramer's rule, the condition in (i) follows.
 \end{proof}

In the light of the characterisation of semi-uniform stability in Theorem \ref{thm:main1}, we shall revisit Theorem \ref{thm:charexpo}. For this, recall $\Psi_{i\omega}$ from \eqref{eq:13HZ} and $\mathcal{Q}(i\omega)$ from Theorem \ref{thm:main1}. 
Note that by Proposition \ref{prop:cd}, the mapping  $\omega\mapsto \Psi_{i\omega}$ and thus $\omega\mapsto \mathcal{Q}(i\omega)$ are continuous.
\begin{thm}
\label{thm:exp2} 
Assume  that $(A, \dom(A))$ as given in (\ref{eq:4HZ}) and (\ref{eq:5HZ}) generates a contraction semigroup $\left(T(t)\right)_{t\geq 0}$ on $X$ with $N=1$. Assume that $\sup_{\omega\in \R}\|\Psi_{i\omega} \|<\infty$.

 Then the following conditions are equivalent:
 \begin{enumerate}
   \item[(i)] $(T(t))_{t\geq 0}$ is exponentially stable;
   \item[(ii)] for all sequences $(\omega_k)_{k \in {\mathbb N}}$ in $\R$, the set
   \[
      \{ v_a \in {\mathbb C}^{n} \mid \liminf_{k\to\infty}\big(\|{\mathcal Q}(i\omega_k)v_a\|+\|v_a^* \left[\begin{smallmatrix} \Psi_{i\omega_k} (b)^* & I_{n} \end{smallmatrix}\right]  \left[\begin{smallmatrix} Q & 0 \\ 0 & -Q \end{smallmatrix} \right] \left[\begin{smallmatrix} \Psi_{i\omega_k} (b) \\ I_{n} \end{smallmatrix}\right] v_a\| \big)=0\}
      \]
      contains only the zero element. 
      \item[(iii)] for all sequences $(\omega_k)_{k \in {\mathbb N}}$ in $\R$, the set
   \[
     \{ v_a \in {\mathbb C}^{n} \mid  \liminf_{k\to\infty}\|{\mathcal Q}(i\omega_k)v_a\|=0 \}
   \]
   contains only the zero element. 
 \end{enumerate}
\end{thm}
\begin{proof}
Note that, by Theorem \ref{thm:charexpo}, we see that $(T(t))_{t\geq 0}$ is exponentially stable if and only if for all $\omega\in \R$, $\mathcal{Q}(i\omega)$ is invertible and $\sup_{\omega\in \R}\|\mathcal{Q}(i\omega)^{-1}\|<\infty$. Since $\Psi_{i\cdot} : {\mathbb R}  \mapsto {\mathbb C}^{n \times n}$ is bounded by assumption, it follows from its definition that the function ${\mathcal Q}$ is bounded as well, see Theorem \ref{thm:main1}(iv).
 Next we turn to the proof of the actual theorem. 
\smallskip

(i)$\Rightarrow$(ii). Note that by our observation from the beginning $\mathcal{Q}(i\omega)$ is invertible for all $\omega\in \R$. Hence, by Theorem \ref{thm:main1}, 
\[
\{ v_a \in {\mathbb C}^{n} \mid \|{\mathcal Q}(i\omega)v_a\|+\|v_a^* \left[\begin{smallmatrix} \Psi_{i\omega} (b)^* & I_{n} \end{smallmatrix}\right]  \left[\begin{smallmatrix} Q & 0 \\ 0 & -Q \end{smallmatrix} \right] \left[\begin{smallmatrix} \Psi_{i\omega} (b) \\ I_{n} \end{smallmatrix}\right] v_a\| =0\}=\{0\}
\]for all $\omega\in \R$. 

Now, let  $(\omega_k)_{k \in {\mathbb N}}$ be a bounded sequence in $\R$ and $v_a\in \R^n$. Then
\[
 \liminf_{k\to\infty}\big(\|{\mathcal Q}(i\omega_k)v_a\|+\|v_a^* \left[\begin{smallmatrix} \Psi_{i\omega_k} (b)^* & I_{n} \end{smallmatrix}\right]  \left[\begin{smallmatrix} Q & 0 \\ 0 & -Q \end{smallmatrix} \right] \left[\begin{smallmatrix} \Psi_{i\omega_k} (b) \\ I_{n} \end{smallmatrix}\right] v_a\| \big)=0
\]
is equivalent to
\[
\|{\mathcal Q}(i\omega)v_a\|+\|v_a^* \left[\begin{smallmatrix} \Psi_{i\omega} (b)^* & I_{n} \end{smallmatrix}\right]  \left[\begin{smallmatrix} Q & 0 \\ 0 & -Q \end{smallmatrix} \right] \left[\begin{smallmatrix} \Psi_{i\omega} (b) \\ I_{n} \end{smallmatrix}\right] v_a\| =0
\]
for one accumulation value $\omega$ of $(\omega_k)_{k \in {\mathbb N}}$. Since the latter implies $v_a=0$, we obtain condition (ii) for all bounded sequences. If $(\omega_k)_{k \in {\mathbb N}}$ is unbounded, we may assume without loss of generality that $(\omega_k)_{k \in {\mathbb N}}\to\infty$ (any bounded subsequence has been dealt with already, and the case  $(\omega_k)_{k \in {\mathbb N}}\to-\infty$ is similar). Since by exponential stability, we get $\sup_{\omega\in \R}\|\mathcal{Q}(i\omega)^{-1}\|<\infty$, Proposition  \ref{prop:unifinv} yields for every non-zero $v_a$ that $\liminf_{k\to\infty} \|\mathcal{Q}(i\omega)v_a\|>0$ and so condition (ii) holds for all sequences $(\omega_k)_{k \in {\mathbb N}}$.

(ii)$\Rightarrow$(iii). Let $v_a$ be such that $\liminf_{k\to\infty} \|\mathcal{Q}(i\omega)v_a\|=0$ for some sequence $(\omega_k)_{k \in {\mathbb N}}$. Then by the reformulation as in the proof of Theorem \ref{thm:main1} (iv)$\Leftrightarrow$(v), we deduce that 
\[
  \liminf_{k\to\infty}\big(\|{\mathcal Q}(i\omega_k)v_a\|+\|v_a^* \left[\begin{smallmatrix} \Psi_{i\omega_k} (b)^* & I_{n} \end{smallmatrix}\right]  \left[\begin{smallmatrix} Q & 0 \\ 0 & -Q \end{smallmatrix} \right] \left[\begin{smallmatrix} \Psi_{i\omega_k} (b) \\ I_{n} \end{smallmatrix}\right] v_a\| =0.
\]
For this particularly note that the reformulation in the proof of Theorem \ref{thm:main1} (iv)$\Leftrightarrow$(v) yields that given a sequence $(\omega_k)_{k\in \N}$ with $\lim_{k\to\infty}\big(\|{\mathcal Q}(i\omega_k)v_a\|$ we obtain $\lim_{k\to\infty}\|v_a^* \left[\begin{smallmatrix} \Psi_{i\omega_k} (b)^* & I_{n} \end{smallmatrix}\right]  \left[\begin{smallmatrix} Q & 0 \\ 0 & -Q \end{smallmatrix} \right] \left[\begin{smallmatrix} \Psi_{i\omega_k} (b) \\ I_{n} \end{smallmatrix}\right] v_a\|=0$ for the \emph{same} sequence.
Hence, $v_a=0$, by (ii).

(iii)$\Rightarrow$(i).  Considering constant sequences, we deduce that (ii) implies Theorem \ref{thm:main1} (iv). Thus, $\mathcal{Q}(i\omega)$ is invertible for all $\omega\in \R$. Finally, by Proposition \ref{prop:unifinv}, we deduce that $\sup_{\omega\in \R}\|\mathcal{Q}(i\omega)^{-1}\|<\infty$.
\end{proof}

\section{Example, network of vibrating strings}
\label{Sec:3}

In this section we study in detail the stability of a (small) network of vibrating strings. The network consists  of three interconnected (undamped) vibrating strings, which are connected to a damper, see also Figure \ref{fig:1}. 

  \begin{figure}[htb]
    \centering
        \includegraphics[scale=0.8]{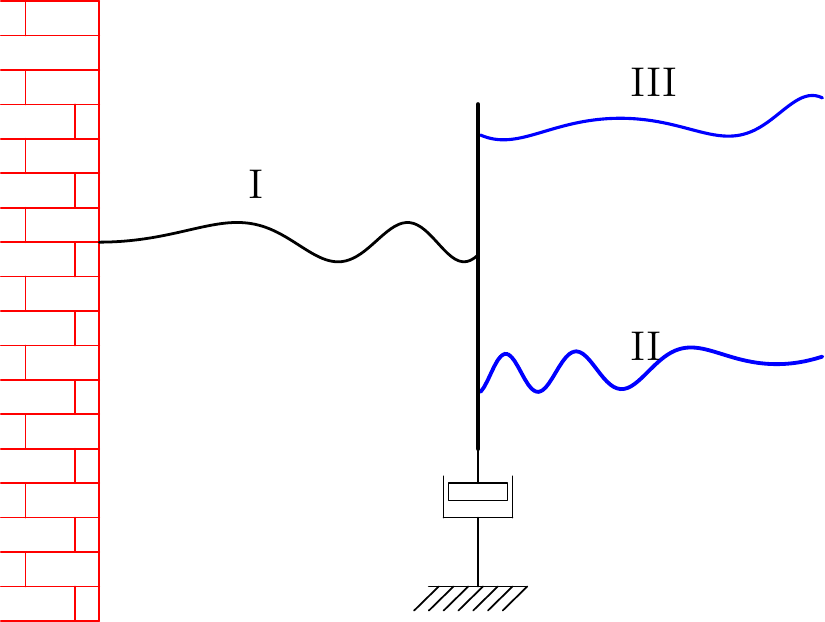}
    \caption{Coupled vibrating strings with damping}
    \label{fig:1}
  \end{figure}

The model of every string is described using the state variables $\rho_k \frac{\partial w_k}{\partial t}$ (the momentum) and $\frac{\partial w_k}{\partial \zeta}$ (the strain) and is given by
\begin{align}
  \frac{\partial x_k}{\partial t}(\zeta,t) = \frac{\partial }{\partial t} \begin{bmatrix} \rho_k \frac{\partial w_k}{\partial t} \\  \frac{\partial w_k}{\partial \zeta} \end{bmatrix} (\zeta,t) &= \begin{bmatrix} 0 & 1 \\ 1 & 0 \end{bmatrix} \frac{\partial}{\partial \zeta} \left(\begin{bmatrix} \frac{1}{\rho_k(\zeta)} & 0 \\ 0 & T_k(\zeta) \end{bmatrix}x_k(\zeta,t)\right)\\
   &= P_{1,k} ({\mathcal H}_kx_k)(\zeta,t), \qquad k\in \{\one,\two,\three\}.
\end{align}
In string $\one$, we let the $\zeta$ run from left to right, but in strings $\two$ and $\three$, we let it go in the opposite direction. We have the following six boundary conditions, see also Figure \ref{fig:1},
\begin{align}
 \label{eq:17HZ}
  \frac{\partial w_{\one}}{\partial t}(a,t) =&\ 0\\
  \label{eq:18HZ}
  T_{\two}(a)\frac{\partial w_{\two}}{\partial \zeta}(a,t) =&\ T_{\three}(a)\frac{\partial w_{\three}}{\partial \zeta}(a,t) =0\\  \label{eq:19HZ}
  \frac{\partial w_{\one}}{\partial t}(b,t) =&\ \frac{\partial w_{\two}}{\partial t}(b,t) = \frac{\partial w_{\three}}{\partial t}(b,t)\\  \label{eq:20HZ}
  -\beta \frac{\partial w_{\one}}{\partial t}(b,t) =&\ T_{\one}(b)\frac{\partial w_{\one}}{\partial \zeta}(b,t) + T_{\two}(b)\frac{\partial w_{\two}}{\partial \zeta}(b,t) + T_{\three}(b)\frac{\partial w_{\three}}{\partial \zeta}(b,t),
\end{align}
where $\beta>0$ denotes the damping coefficient of the damper. 

Formulating the system as (\ref{eq:1HZ}) is done by putting the 3 states into one 6-dimensional vector. For this system the $P_1$ and ${\mathcal H}$ are given as the block-diagonal matrices with the three $P_{1,k}$'s and ${\mathcal H}_{k}$'s on the diagonal, respectively, that is,
\begin{equation}
\label{eq:P1H}
   P_1 = \begin{bmatrix}P_{1,\one} & & \\ & P_{1,\two} & \\ & & P_{1,\three} \end{bmatrix}\text{ and } \mathcal{H} = \begin{bmatrix}\mathcal{H}_{\one} & & \\ & \mathcal{H}_{\two} & \\ & & \mathcal{H}_{\three} \end{bmatrix}.
\end{equation}
We investigate the asymptotic stability of the above model by employing part (v) of Theorem \ref{thm:main1}. For this we need the solution of the differential equation (\ref{eq:12HZ})--(\ref{eq:12aHZ}). Using that $N=1$ and the special form of ${\mathcal H}$ and $P_1$, this can be written as 3 differential equations on $[a,b]$ as
\begin{equation}
\label{eq:21HZ}
  \frac{dv_k}{d\zeta}(\zeta) = i\omega \begin{bmatrix} 0 & T_k(\zeta)^{-1} \\ \rho_k(\zeta) & 0 \end{bmatrix} v_k(\zeta),\mbox { with } v_k(a) = v_{a,k}, \quad k\in \{{\one},\two,\three\}.
\end{equation}
The boundary conditions of (\ref{eq:17HZ}) and (\ref{eq:18HZ}) give that
\begin{equation}
\label{eq:22HZ}
  v_{a,{\one}} = \begin{bmatrix} 0 \\ T_{\one}(a) \frac{\partial w_{\one}}{\partial \zeta}(a) \end{bmatrix}:=\begin{bmatrix} 0 \\ F_{\one}(a) \end{bmatrix}, \quad  v_{a,k} = \begin{bmatrix}  \frac{\partial w_k}{\partial t}(a) \\ 0 \end{bmatrix}:=\begin{bmatrix} \nu_{a,k} \\ 0 \end{bmatrix},\quad k \in  \{\two,\three\}.
\end{equation}
The other three boundary conditions imply that
\begin{equation}
\label{eq:23HZ}
  v_{k}(b) = \begin{bmatrix} \frac{\partial w_k}{\partial t}(b) \\ T_k(b) \frac{\partial w_k}{\partial \zeta}(b) \end{bmatrix}:=\begin{bmatrix} \nu_b \\ F_k(b) \end{bmatrix},\quad  k \in  \{{\one},\two,\three\}.
\end{equation}
with $-\beta \nu_b = F_{\one}(b) + F_{\two}(b) + F_{\three}(b)$.

Having provided the reformulation of  (\ref{eq:12HZ})--(\ref{eq:12aHZ}), for convenience of the reader, we provide the consequence of our main theorem in this particular situation. For this, we quickly check the generation property of the port-Hamiltonian operator.
\begin{thm}
\label{thm:genexampl} 
For $k\in \{\one,\two,\three\}$, let $T_k,\rho_k \in L_\infty(a,b)$ and assume there exists $\varepsilon_0>0$ such that for a.e.~$\zeta\in (a,b)$, $T_k(\zeta),\rho_k(\zeta)\geq \varepsilon_0$. With \eqref{eq:P1H} in $L_{2,\mathcal{H}}((a,b); {\mathbb C}^6)$ consider the operator
\[
   A = P_1\frac{d}{d\zeta}\mathcal{H}
\]
with domain
\[
\dom(A) = \{ x\in L_{2,\mathcal{H}}((a,b); {\mathbb C}^6) \mid  \mathcal{H}x \in H^1((a,b);{\mathbb C}^6), W_B \begin{pmatrix} (\mathcal{H}x)(b) \\ (\mathcal{H}x)(a) \end{pmatrix} = 0\},
\]where $W_B$ is given by  \eqref{eq:17HZ}--\eqref{eq:20HZ}, that is,
\[
   W_B = \left(\begin{array}{ccccccccccccc} 1 & 0 & -1 & 0 & 0 & 0 &\vline&  0 & 0 & 0 & 0 & 0 & 0 \\ 
                                        1 & 0 & 0 & 0 & -1 & 0 &\vline& 0 & 0 & 0 & 0 & 0 & 0 \\ 
                                        \beta & 1 & 0 & 1 & 0 & 1 &\vline& 0 & 0 & 0 & 0 & 0 & 0 \\ \hline 
                                        0 & 0 & 0 & 0 & 0 & 0 &\vline& 1 & 0 & 0 & 0 & 0 & 0 \\ 
                                        0 & 0 & 0 & 0 & 0 & 0 &\vline& 0 & 0 & 0 & 1 & 0 & 0 \\ 
                                        0 & 0 & 0 & 0 & 0 & 0 &\vline& 0 & 0 & 0 & 0 & 0 & 1 \\ 
                                        \end{array}\right).
\] 
Then $(A,\dom(A))$ is the generator of contraction semigroup.
\end{thm}
\begin{proof}
 We apply Theorem \ref{thm:gencon}. For this, it is easy to see that $W_B$ has full rank. Hence, using $Q=P_1$ and $R_{ext} = \frac{1}{\sqrt{2}} \begin{bmatrix} P_1 & - P_1 \\ I_{6} & I_{6} \end{bmatrix}$, we need to establish inequality \eqref{eq:WB}. For this, using $P_1^2 = I_6$ and
 \[
     R_{ext}^{-1} = \frac{1}{\sqrt{2}} \begin{bmatrix} P_1 &I_6\\  - P_1  & I_{6} \end{bmatrix},
 \] we compute
 \[
      R_{ext}^{-1} \begin{bmatrix} 0 & I_6 \\ I_6 & 0 \end{bmatrix} R_{ext}^{-*} = \begin{bmatrix} P_1 & 0 \\ 0 & -P_1\end{bmatrix},\]
thus \begin{align*}
& W_BR_{ext}^{-1} \begin{bmatrix} 0 & I_6 \\ I_6 & 0 \end{bmatrix} (W_BR_{ext}^{-1})^*  = W_B \begin{bmatrix} P_1 & 0 \\ 0 & -P_1\end{bmatrix} W_B^* \\
 & = \begin{bmatrix} \begin{bmatrix} 
 \begin{bmatrix} 1 & 0 \\ 1 & 0 \end{bmatrix} &  \begin{bmatrix}  -1 & 0 \\ 0 & 0 \end{bmatrix} &  \begin{bmatrix} 0 & 0 \\ -1 & 0  \end{bmatrix}   \\
   \begin{bmatrix}\beta & 1 \\ 0 & 0 \end{bmatrix} &  \begin{bmatrix} 0 & 1 \\ 0 & 0  \end{bmatrix} &  \begin{bmatrix}  0 & 1 \\ 0 & 0\end{bmatrix}  
\\
   \begin{bmatrix} 0 & 0 \\ 0 & 0 \end{bmatrix} &     \begin{bmatrix} 0 & 0 \\ 0 & 0 \end{bmatrix} &    \begin{bmatrix} 0 & 0 \\ 0 & 0 \end{bmatrix}  \end{bmatrix} & 
   \begin{bmatrix} 
 \begin{bmatrix} 0 & 0 \\ 0 & 0 \end{bmatrix} &  \begin{bmatrix}  0 & 0 \\ 0 & 0 \end{bmatrix} &  \begin{bmatrix} 0 & 0 \\ 0 & 0  \end{bmatrix}   \\
   \begin{bmatrix}0 & 0 \\ 1 & 0 \end{bmatrix} &  \begin{bmatrix} 0 & 0\\ 0 & 0  \end{bmatrix} &  \begin{bmatrix}  0 & 0 \\ 0 & 0\end{bmatrix}  
\\
   \begin{bmatrix} 0 & 0 \\ 0 & 0 \end{bmatrix} &     \begin{bmatrix} 0 & 1 \\ 0 & 0 \end{bmatrix} &    \begin{bmatrix} 0 & 0 \\ 0 & 1 \end{bmatrix}  \end{bmatrix}\end{bmatrix} \times \\
   & \quad \times\begin{bmatrix} P_1 & 0 \\ 0 & -P_1\end{bmatrix} W_B^*  \\
   &= \begin{bmatrix} 
 \begin{bmatrix} 1 & 0 \\ 1 & 0 \end{bmatrix} &  \begin{bmatrix}  -1 & 0 \\ 0 & 0 \end{bmatrix} &  \begin{bmatrix} 0 & 0 \\ -1 & 0  \end{bmatrix}   \\
   \begin{bmatrix}\beta & 1 \\ 0 & 0 \end{bmatrix} &  \begin{bmatrix} 0 & 1 \\ 0 & 0  \end{bmatrix} &  \begin{bmatrix}  0 & 1 \\ 0 & 0\end{bmatrix}  
\\
 0 &   0&  0\end{bmatrix} \begin{bmatrix} P_{1,\one}  & & \\ & P_{1,\two} & \\ & & P_{1,\three}\end{bmatrix} \begin{bmatrix} 
 \begin{bmatrix} 1 & 1 \\ 0 & 0 \end{bmatrix} &  \begin{bmatrix}  \beta & 0 \\ 1 & 0 \end{bmatrix} &  0  \\
   \begin{bmatrix}-1 & 0 \\ 0 & 0 \end{bmatrix} &  \begin{bmatrix} 0 & 0 \\ 1 & 0  \end{bmatrix} &  0\\
   \begin{bmatrix} 0 & -1 \\ 0 & 0 \end{bmatrix} &     \begin{bmatrix}  0& 0 \\ 1 & 0 \end{bmatrix} &   0  \end{bmatrix} \\
   & - \begin{bmatrix} 
0 &  0 &  0  \\
   \begin{bmatrix}0 & 0 \\ 1 & 0 \end{bmatrix} &  0 &  0
\\
0 &     \begin{bmatrix} 0 & 1 \\ 0 & 0 \end{bmatrix} &    \begin{bmatrix} 1 & 0 \\ 0 & 0 \end{bmatrix}  \end{bmatrix} \begin{bmatrix} P_{1,\one}  & & \\ & P_{1,\two} & \\ & & P_{1,\three}\end{bmatrix}\begin{bmatrix} 
 0 &  \begin{bmatrix}  0 & 1 \\ 0 & 0 \end{bmatrix} &  0 \\
  0 & 0 &  \begin{bmatrix}  0 & 0 \\ 1 & 0\end{bmatrix}  
\\
0 &   0 &    \begin{bmatrix} 1 & 0 \\ 0 & 0 \end{bmatrix}  \end{bmatrix} \\
& =  \begin{bmatrix}
0 & 0 & 0 & 0 & 0 & 0 \\
0 & 0 & 0 & 0 & 0 & 0 \\
0 & 0 & 2\beta & 0 & 0 & 0 \\
0 & 0 & 0 & 0 & 0 & 0 \\
0 & 0 & 0 & 0 & 0 & 0 \\
0 & 0 & 0 & 0 & 0 & 0 
\end{bmatrix},
 \end{align*}  
which is non-negative definite as $\beta\geq 0$. Note that we could also have shown this by checking the ineqality in part (iii) of Theorem \ref{thm:gencon}.
\end{proof}

Thus, we are in the position to apply Theorem \ref{thm:main1}. For this, let $\Psi_\lambda \colon [a,b)\to \CC^{6\times 6}$ with $
\Psi_\lambda(a)=I_6$ be the fundamental solution of
\begin{equation}\label{eq:exampleode}
   \frac{d v}{ d\zeta}(\zeta) = \lambda P_1^{-1}\mathcal{H}(\zeta)^{-1}v(\zeta)\quad \text{for }\zeta\in (a,b)\text{ and }\lambda\in \CC.
\end{equation}  
As a result, we deduce the following characterisation of stability.
\begin{thm}
\label{thm:mainex} 
Consider $(A,\dom(A))$ as provided in Theorem \ref{thm:genexampl} and let $(T(t))_{t\geq0}$ be the generated contraction semigroup. Then the following conditions are equivalent:
\begin{enumerate}[(i)]
\item $(T(t))_{t\geq 0}$ is semi-uniformly stable.
\item For all $\omega \in {\mathbb R}$ the set
\[
  \{ v_a \in {\mathbb C}^{6} \mid{W_B\begin{bmatrix} \Psi_{i\omega} (b) \\ I_{6}\end{bmatrix}}v_a=0 \mbox{ and } v_a^* \left[\begin{smallmatrix} \Psi_{i\omega} (b)^* & I_{6} \end{smallmatrix}\right]  \left[\begin{smallmatrix} P_1 & 0 \\ 0 & -P_1 \end{smallmatrix} \right] \left[\begin{smallmatrix} \Psi_{i\omega} (b) \\ I_{6} \end{smallmatrix}\right] v_a =0\}
\]
contains only the zero element. 
\item For all $\omega\in \R$ the following holds: Let $v\colon [a,b]\to \CC^6 = (\CC^2)^3$ be a solution of \eqref{eq:exampleode} with $\lambda=i\omega$. Denote $v(a)= \big(\begin{bmatrix} \nu_{a,k} \\ F_k(a) \end{bmatrix}\big)_{k\in \{\one,\two,\three\}}$ and similarly for $v(b)$. If
\begin{align*}
 &   \nu_{a,\one} =\nu_{b,\one} = \nu_{b,\two}=\nu_{b,\three}=0,   F_{\two}(a)=F_{\three}(a) = 0, F_{\one}(b)+F_{\two}(b)+F_{\three}(b) =0, 
\end{align*}
then $v=0$.
\end{enumerate}
\end{thm}
\begin{proof}
By Theorem \ref{thm:main1} it suffices to show that the latter two items are equivalent. For this note that any solution of \eqref{eq:exampleode} for $\lambda=i\omega$ can be represented by $v(\cdot)= \Psi_{i\omega}(\cdot)v_a$ for some $v_a\in \CC^6$ (and vice versa). In this case, $v(a)=v_a$ and $v(b)= \Psi_{i\omega}(b)v_a$. Hence, it follows that the equations in the set of condition (ii), can be equivalently expressed as
\begin{equation}\label{eq:proofeq2}
    W_B \begin{bmatrix} v(b) \\ v(a) \end{bmatrix} = 0\text{ and } v(b)^*P_1v(b) - v(a)^*P_1v(a) =0.
\end{equation}
To obtain the claim, it thus suffices to show that the equalities in \eqref{eq:proofeq2} are equivalent to the ones claimed in condition (iii). For this note that the conditions encoded in $W_B \left[\begin{smallmatrix} v(b) \\ v(a) \end{smallmatrix} \right]= 0$ read
\begin{align*}
   \nu_{a,I} =&\ 0,\quad  F_{\two}(a)=F_{\three}(a) = 0 \\
   \nu_b\coloneqq&\ \nu_{b,\one} = \nu_{b,\two}=\nu_{b,\three}\\
 -\beta \nu_b=&\  F_{\one}(b)+F_{\two}(b)+F_{\three}(b).
 \end{align*}
 Next, using the form of $P_1$ and $v$ and the boundary conditions, we reformulate  $v(b)^*P_1v(b) - v(a)^*P_1v(a) =0$:
 \begin{align*}
v(b)^*P_1v(b)& - v(a)^*P_1v(a) \\ =&\ 2\sum_{k\in \{\one,\two,\three\}} \big( \begin{bmatrix} \nu_{b,k} \\ F_k(b) \end{bmatrix}^*\begin{bmatrix} 0 & 1\\ 1 & 0\end{bmatrix}\begin{bmatrix} \nu_{b,k} \\ F_k(b) \end{bmatrix}-\begin{bmatrix} \nu_{a,k} \\ F_k(a) \end{bmatrix}^*\begin{bmatrix} 0 & 1\\ 1 & 0\end{bmatrix}\begin{bmatrix} \nu_{a,k} \\ F_k(a) \end{bmatrix}  \big) \\
 =&\  2\sum_{k\in \{\one,\two,\three\}} \Re \nu_{b,k}^*F_{k}(b) - \Re \nu_{a,k}^*F_{k}(a) = \sum_{k\in \{\one,\two,\three\}} \Re \nu_{b,k}^*F_{k}(b) \\
 =&\ 2\Re\nu_b^*( F_{\one}(b)+F_{\two}(b)+F_{\three}(b)).
 \end{align*}
 Hence, $v(b)^*P_1v(b) - v(a)^*P_1v(a) =0$ is the same as saying
 \[
  \Re\nu_b^* (F_{\one}(b) + F_{\two}(b) + F_{\three}(b)) =0.
 \]
 Combining this with the equality $-\beta \nu_b= F_{\one}(b)+F_{\two}(b)+F_{\three}(b)$ we see that this
 is equivalent to
 \[
 \nu_b =0 \text{ and } F_{\one}(b)+F_{\two}(b)+F_{\three}(b)=0.
 \]
 Indeed, sufficiency of the latter system for the former being obvious; the necessity of the latter follows upon substituting the first equation of the former into its second equation to obtain $-\beta\Re\nu_b^*\nu_b=0$ leading to $\nu_b=0$ as $\beta>0$ and, hence, $F_{\one}(b)+F_{\two}(b)+F_{\three}(b)=0$ by the first equation.
\end{proof}

\begin{rem}\label{rem:eig} In the situation of Theorem \ref{thm:mainex} assume that there exists $\omega\in \R$ such that in condition (iii), we find $v\neq 0$ satisfying the equations for the boundary values of $v$. Then $i\omega$ is an eigenvalue of $A$. Indeed, this follows from Lemma \ref{L:eig} together with the equivalence of (iv) and (v) in Theorem \ref{thm:main1} and the respective proof.
\end{rem}

Next, we provide necessary conditions for the validity of condition (iii) in Theorem \ref{thm:mainex}. In any case, challenges only arise for $\omega\neq 0$, as the next result confirms.
\begin{lem}
\label{lem:omegazero} 
Let $v\colon [a,b]\to \CC^6 = (\CC^2)^3$ be a solution of \eqref{eq:exampleode} with $\lambda=0$. Denote $v(a)= \big(\begin{bmatrix} \nu_{a,k} \\ F_k(a) \end{bmatrix}\big)_{k\in \{\one,\two,\three\}}$ and similarly for $v(b)$. If
\begin{align*}
 &   \nu_{a,I} =\nu_{b,\one} = \nu_{b,\two}=\nu_{b,\three}=0,   F_{\two}(a)=F_{\three}(a) = 0, F_{\one}(b)+F_{\two}(b)+F_{\three}(b) =0, 
\end{align*}
then $v=0$.
\end{lem}
\begin{proof}
It is easy to see that for $\omega=0$ the solution of (\ref{eq:21HZ}) are constants, i.e.,
\[
  v_{k}(\zeta) =\begin{bmatrix} \nu_k \\ F_k \end{bmatrix},\quad \zeta \in [a,b],\quad  k \in  \{{\one},\two,\three\}.
\]
Since, $\nu_k = \nu_{a,k}=\nu_{b,k}$ for all $k \in  \{{\one},\two,\three\}$, by assumption it follows that $\nu_k=0$ for all $k \in  \{{\one},\two,\three\}$. Similarly, we infer $F_k=0$ for $k\in\{\two,\three\}$ and from $ F_{\one}(b)+F_{\two}(b)+F_{\three}(b) =0$ it thus follows that $F_{\three}=0$, which leads to $v=0$.
\end{proof}

For the next observation, we recall that, by \eqref{eq:P1H}, the matrices $P_1$ and $\mathcal{H}(\zeta)$ are block-diagonal with $2\times 2$-blocks sitting in its diagonal. Thus, the fundamental solution of \eqref{eq:exampleode} is, too, block-diagonal. Hence, for any $\omega\in \R$, $\Psi_{i\omega}$ can be represented by 
\begin{equation}
\label{eq:psik}
\Psi_{i\omega,k}(\zeta) = \begin{bmatrix} \Psi_{i\omega,k}^{11}(\zeta) & \Psi_{i\omega,k}^{12}(\zeta)\\ \Psi_{i\omega,k}^{21} (\zeta) & \Psi_{i\omega,k}^{22}(\zeta) \end{bmatrix},\quad k\in \{\one,\two,\three\},
\end{equation} each $\Psi_{i\omega,k}$ being the fundamental solution of \eqref{eq:21HZ}, that is,
\begin{equation}
\label{eq:psikode}
  \frac{dv_k}{d\zeta}(\zeta) = i\omega \begin{bmatrix} 0 & T_k(\zeta)^{-1} \\ \rho_k(\zeta) & 0 \end{bmatrix} v_k(\zeta), \quad k\in \{{\one},\two,\three\}.
\end{equation}
With the latter in mind, we find a criterion eventually ensuring the validity of condition (iii) in Theorem \ref{thm:mainex}.
\begin{prop}\label{prop:omeganonzero} Let $\omega\in \R$ and let $v\colon [a,b]\to \CC^6$ be a solution of \eqref{eq:exampleode} with $\lambda=i\omega$. Denote $v(a)= \big(\begin{bmatrix} \nu_{a,k} \\ F_k(a) \end{bmatrix}\big)_{k\in \{\one,\two,\three\}}$ and similarly for $v(b)$ and assume
\begin{align*}
 &   \nu_{a,I} =\nu_{b,\one} = \nu_{b,\two}=\nu_{b,\three}=0,   F_{\two}(a)=F_{\three}(a) = 0, F_{\one}(b)+F_{\two}(b)+F_{\three}(b) =0.
\end{align*}
If 
\begin{equation}\label{eq:psi}
    |\Psi_{i\omega,{\one}}^{12}(b)\Psi_{i\omega,\two}^{11}(b)|+|\Psi_{i\omega,{\one}}^{12}(b)\Psi_{i\omega,\three}^{11}(b)|+|\Psi_{i\omega,{\two}}^{11}(b)\Psi_{i\omega,\three}^{11}(b)|  > 0,
\end{equation}
then $v=0$.
\end{prop}
\begin{proof}
By the mentioned block structure, it follows that we can write 
\[
   v_k(\zeta) = \Psi_{i\omega,k}(\zeta) v_{a,k} = \begin{bmatrix} \Psi_{i\omega,k}^{11}(\zeta) & \Psi_{i\omega,k}^{12}(\zeta)\\ \Psi_{i\omega,k}^{21} (\zeta) & \Psi_{i\omega,k}^{22}(\zeta) \end{bmatrix} \begin{bmatrix} \nu_{a,k} \\ F_k(a) \end{bmatrix}, \quad k \in \{{\one},\two,\three\}.
\]
In particular,
\[
\begin{bmatrix} \nu_{b,k} \\ F_k(b) \end{bmatrix} =  \begin{bmatrix} \Psi_{i\omega,k}^{11}(b) & \Psi_{i\omega,k}^{12}(b)\\ \Psi_{i\omega,k}^{21} (b) & \Psi_{i\omega,k}^{22}(b) \end{bmatrix} \begin{bmatrix} \nu_{a,k} \\ F_k(a) \end{bmatrix}.
\]
By assumption, we obtain
\begin{align*}
0&=\nu_{b,\one} = \Psi_{i\omega,\one}^{11}(b)\nu_{a,\one} + \Psi_{i\omega,\one}^{12}(b)F_\one(a) =\Psi_{i\omega,\one}^{12}(b)F_\one(a), \\
0&=\nu_{b,\two} = \Psi_{i\omega,\two}^{11}(b)\nu_{a,\two} + \Psi_{i\omega,\two}^{12}(b)F_\two(a) =\Psi_{i\omega,\two}^{11}(b)\nu_{a,\two}, \\
0&=\nu_{b,\three} = \Psi_{i\omega,\three}^{11}(b)\nu_{a,\three} + \Psi_{i\omega,\two}^{12}(b)F_\three(a) =\Psi_{i\omega,\three}^{11}(b)\nu_{a,\three}.
\end{align*}
Using the condition in \eqref{eq:psi}, we obtain one of the following alternatives:
\[
\nu_{a,\two}=\nu_{a,\three}=0\text{ or }F_{\one}(a)=\nu_{b,\two}=0\text{ or }F_{\one}(a)=\nu_{b,\three}=0.
\]
We consider $\nu_{a,\two}=\nu_{a,\three}=0$ first. For this note that, by assumption, $F_{\two}(a)=F_{\three}(a) = 0$ and, so, $v_k =0$ for $k\in \{\two,\three\}$. In particular, we find $F_{\two}(b) = F_{\three}(b)=0$. Thus, the assumed equality $F_{\one}(b)+F_{\two}(b)+F_{\three}(b) =0$ leads to $F_{\one}(b)=0$ and, in conjunction with $\nu_{b,\one}=0$, leads to $v_{\one}=0$ by uniqueness of the solution of \eqref{eq:psikode} for given data.

Next, we consider $F_{\one}(a)=0$ and $\nu_{a,\two}=0\text{ or }\nu_{a,\three}=0$. It suffices to consider the first alternative as the other case follows a similar rationale. In any case, $F_{\one}(a)=0$ together with the assumed $\nu_{a,\one}=0$ implies $v_{\one}=0$. The condition $\nu_{a,\two}=0$ together with $F_{\two}(a)=0$ leads  to $v_{\two}=0$. As a consequence, we find $F_{\one}(b)=F_{\two}(b)=0$. Hence, by assumption, $F_{\three}(b)=0$, which together with $v_{b,\three}=0$ leads to $v_{\three}=0$ again by a uniqueness argument.
\end{proof}
We now have a criterion at hand that characterises semi-uniform (or, equivalently, asymptotic) stability for the semigroup $\left(T(t)\right)_{t\geq 0}$ in question:
\begin{thm}
\label{thm:stabgen} 
For $\omega\in \R$ let $\Psi_{i\omega,k}$ be as in \eqref{eq:psik}, $k\in \{\one,\two,\three\}$.
If for all $\omega\in \R\setminus\{0\}$, we have
\[
    |\Psi_{i\omega,{\one}}^{12}(b)\Psi_{i\omega,\two}^{11}(b)|+|\Psi_{i\omega,{\one}}^{12}(b)\Psi_{i\omega,\three}^{11}(b)|+|\Psi_{i\omega,{\two}}^{11}(b)\Psi_{i\omega,\three}^{11}(b)|  > 0,
\]then the semigroup $(T(t))_{t \geq 0}$ given in Theorem \ref{thm:genexampl} is semi-uniformly stable.

If, on the other hand, 
\[
    |\Psi_{i\omega,{\one}}^{12}(b)\Psi_{i\omega,\two}^{11}(b)|+|\Psi_{i\omega,{\one}}^{12}(b)\Psi_{i\omega,\three}^{11}(b)|+|\Psi_{i\omega,{\two}}^{11}(b)\Psi_{i\omega,\three}^{11}(b)|  =0
\]
for one $\omega\in \R$, then $i\omega$ is an eigenvalue of $(A,\dom(A))$ given in  Theorem \ref{thm:genexampl} and $(T(t))_{t\geq 0}$ is not semi-uniformly stable.
\end{thm}
\begin{proof}
The result follows upon application of Theorem \ref{thm:mainex}. In the first case, by Lemma \ref{lem:omegazero} (for $\omega=0$) and Proposition \ref{prop:omeganonzero} (for $\omega\neq 0$) condition (iii)~of Theorem \ref{thm:mainex} is satisfied and, therefore, $(T(t))_{t\geq 0}$ is semi-uniformly stable.

Conversely, assume 
\[
    |\Psi_{i\omega,{\one}}^{12}(b)\Psi_{i\omega,\two}^{11}(b)|+|\Psi_{i\omega,{\one}}^{12}(b)\Psi_{i\omega,\three}^{11}(b)|+|\Psi_{i\omega,{\two}}^{11}(b)\Psi_{i\omega,\three}^{11}(b)|  =0
\]
for one $\omega\in \R$. By Remark \ref{rem:eig} it suffices to construct a non-zero solution $v$ satisfying the boundary equations in condition (iii)~of Theorem \ref{thm:mainex}. The assumption implies that two of the three numbers $\Psi_{i\omega,{\one}}^{12}(b), \Psi_{i\omega,\two}^{11}(b), \Psi_{i\omega,\three}^{11}(b)$ are zero. We discuss the cases $\Psi_{i\omega,{\one}}^{12}(b)= \Psi_{i\omega,\two}^{11}(b)=0$ and $\Psi_{i\omega,{\three}}^{11}(b)= \Psi_{i\omega,\two}^{11}(b)=0$, the remaining one can be dealt with similarly to the first. 

Consider the first case, i.e., $\Psi_{i\omega,{\one}}^{12}(b)= \Psi_{i\omega,\two}^{11}(b)=0$. Since $\Psi_{i\omega,\one}$ is a fundamental solution, $\Psi_{i\omega,\one}(b)$ is invertible, thus, by $\Psi_{i\omega,{\one}}^{12}(b)=0$, it follows that $\Psi_{i\omega,{\one}}^{22}(b)\neq 0$. For the construction of a solution, it suffices to provide a non-trivial initial value $v(a)= \big(\left[\begin{smallmatrix} \nu_{a,k} \\ F_k(a) \end{smallmatrix}\right]\big)_{k\in \{\one,\two,\three\}}$ which leads to the satisfaction of the boundary conditions in condition (iii) of Theorem  \ref{thm:mainex}. For this, we put $\nu_{a,\one}=\nu_{a,\three}=0$ as well as $F_{\two}(a)=F_{\three}(a)=0$. Moreover, we set $\nu_{a,\two}=1$ and $F_{\one}(a)= (\Psi_{i\omega,\one}^{22}(b))^{-1}(-\Psi_{i\omega,\two}^{21}(b))$. Then, $v(a)\neq 0$ and, hence, $v\neq 0$. Moreover, evidently, $\nu_{b,\three}=F_{\three}(b)=0$ since $v_{\three}=0$. Next,
\begin{align*}
   \nu_{b,\one}&= \Psi_{i\omega,\one}^{11}(b)\nu_{a,\one} + \Psi_{i\omega,\one}^{12}(b)F_\one(a)=0\\
   \nu_{b,\two}& = \Psi_{i\omega,\two}^{11}(b)\nu_{a,\two} + \Psi_{i\omega,\two}^{12}(b)F_\two(a)=0.
\end{align*}
It remains to check whether $F_{\one}(b)+F_{\two}(b)+F_{\three}(b)=0$. The last term of the left-hand side being zero, we compute
\begin{align*}
   F_{\one}(b)+F_{\two}(b) & = \Psi_{i\omega,\one}^{21}(b)\nu_{a,\one} + \Psi_{i\omega,\one}^{22}(b)F_\one(a)+\Psi_{i\omega,\two}^{21}(b)\nu_{a,\two} + \Psi_{i\omega,\two}^{22}(b)F_\two(a) \\ &= 0 + \Psi_{i\omega,\one}^{22}(b)((\Psi_{i\omega,\one}^{22}(b))^{-1}(-\Psi_{i\omega,\two}^{21}(b))) +\Psi_{i\omega,\two}^{21}(b) + 0=0,
\end{align*}
which establishes all boundary conditions satisfied by $v\neq 0$.
%

Finally, consider the case $\Psi_{i\omega,{\three}}^{11}(b)= \Psi_{i\omega,\two}^{11}(b)=0$. Since $\Psi_{i\omega,{\three}}$ is a fundamental solution, $\Psi_{i\omega,{\three}}^{11}(b)=0$ implies $\Psi_{i\omega,{\three}}^{21}(b)\neq 0$. With this observation, we may choose $\nu_{a,\one}=F_{\one}(a)=F_{\two}(a)=F_{\three}(a)=0$ and $\nu_{a,\two}=1$ and $\nu_{a,\three}= -(\Psi_{i\omega,\three}^{21}(b))^{-1}\Psi_{i\omega,\two}^{21}(b)$. Then $v(a)\neq 0$ and, thus, $v\neq 0$. Also, it is not difficult to see that $\nu_{b,k}=0$ for all $k\in \{\one,\two,\three\}$. Moreover, $F_{\one}(b)=0$. It remains to check $F_{\two}(b)+F_{\three}(b)=0$. For this, we compute
\begin{align*}
 F_{\two}(b)+F_{\three}(b) & = \Psi_{i\omega,\two}^{21}(b)\nu_{a,\two} + \Psi_{i\omega,\two}^{22}(b)F_\two(a)+\Psi_{i\omega,\three}^{21}(b)\nu_{a,\three} + \Psi_{i\omega,\three}^{22}(b)F_\three(a) \\
 & = \Psi_{i\omega,\two}^{21}(b) + 0  -\Psi_{i\omega,\three}^{21}(b)(\Psi_{i\omega,\three}^{21}(b))^{-1}\Psi_{i\omega,\two}^{21}(b)+ 0 = 0,
\end{align*}
which proves the assertion.
\end{proof}

Theorem \ref{thm:stabgen} is formulated in terms of the fundamental solution of \eqref{eq:exampleode}. Since this fundamental solution is a rather complicated object for variable coefficients $\mathcal{H}(\zeta)$ a more detailed characterisation can only be provided if we assume additional properties of $\mathcal{H}$. As an exemplary case, we consider the most elementary one next, the case of constant coefficients. We shall see that already in this case the conditions characterising asymptotic stability are quite subtle.

For the rest of this section, we assume that $\mathcal{H}$ is constant, that is,
\begin{equation}\label{eq:Hconst}
   \mathcal{H}_{k}(\zeta)=\begin{bmatrix}\frac{1}{\rho_k} & 0 \\ 0 & T_k\end{bmatrix}
\end{equation}
for some $\rho_k, T_k>0$ for all $k\in \{\one,\two,\three\}$. The criterion in Theorem \ref{thm:stabgen} can be expressed in terms of certain ratios of the coefficients. Therefore for $k\in \{\one,\two,\three\}$ we introduce $c_k>0$ satisfying 
\begin{equation}\label{eq:ck}
c_k^2=\frac{\rho_k}{T_k}.
\end{equation} 
Note that $c_k$ is the inverse of the characteristic speed in the strings. By differentiation and uniqueness of fundamental solutions, it is not difficult to see that for $k\in \{\one,\two,\three\}$ we have
\begin{equation}\label{eq:constpsi}
  \Psi_{i\omega, k}(\zeta) = \begin{bmatrix} \cos(c_k \omega(\zeta-a)) & \frac{ic_k}{\rho_k} \sin(c_k\omega(\zeta-a)) \\  ic_k T_k \sin(c_k\omega(\zeta-a)) & \cos(c_k \omega(\zeta-a))
  \end{bmatrix}.
\end{equation}
Thus
\begin{align}
&\label{eq:psi1}  \Psi_{i\omega,{\one}}^{12}(b) = (ic_\one/\rho_\one)\sin(c_{\one}\, \omega(b-a))\\
&\label{eq:psi2}  \Psi_{i\omega,{\two}}^{11}(b) = \cos(c_{\two}\omega(b-a)) \\
&\label{eq:psi3}  \Psi_{i\omega,{\three}}^{11}(b) =  \cos(c_{\three}  \omega(b-a)) .
\end{align} 
For the next result we introduce
\begin{align*}
\mathbb{O}/\mathbb{O} & \coloneqq \{\frac{1+2k}{1+2\ell}; k,\ell\in\Z\} \\
\mathbb{E}/\mathbb{O} & \coloneqq \{\frac{2k}{1+2\ell}; k,\ell\in\Z\}.
\end{align*}
\begin{prop}\label{prop:instabpart} 
For $\omega\in \R$ consider \eqref{eq:constpsi} and let $c_k>0$ be given by \eqref{eq:ck} for $k\in \{\one,\two,\three\}$. Then the following conditions are equivalent:
\begin{enumerate}[(i)]
\item There is $\omega\in \R$ with
\[
    |\Psi_{i\omega,{\one}}^{12}(b)\Psi_{i\omega,\two}^{11}(b)|+|\Psi_{i\omega,{\one}}^{12}(b)\Psi_{i\omega,\three}^{11}(b)|+|\Psi_{i\omega,{\two}}^{11}(b)\Psi_{i\omega,\three}^{11}(b)| =0.
\]
\item At least one of the following statements is true
\[
    \frac{c_{\one}}{c_{\two}} \in \mathbb{E}/\mathbb{O},\quad     \frac{c_{\one}}{c_{\three}} \in \mathbb{E}/\mathbb{O},\quad     \frac{c_{\two}}{c_{\three}} \in \mathbb{O}/\mathbb{O}.
\]
\end{enumerate}
In either case, there exists infinitely many $\omega \in \R$ with the property in (i).
\end{prop}
\begin{proof}
Before we dive into the actual proof, we deduce the following equivalences from the form of the fundamental solutions given in \eqref{eq:psi1}--(\ref{eq:psi3}), 
\begin{align*}
&  \Psi_{i\omega,{\one}}^{12}(b) = 0 \iff \sin(c_{\one}\, \omega(b-a))=0\iff \omega \in \frac{\pi}{c_{\one}(b-a)} \Z \\
&  \Psi_{i\omega,{\two}}^{11}(b) = 0 \iff \cos(c_{\two}\omega(b-a)) =0\iff \omega \in  \frac{\pi}{2c_{\two}(b-a)}+\frac{\pi}{c_{\two}(b-a)} \Z\\
&  \Psi_{i\omega,{\three}}^{11}(b) = 0 \iff \cos(c_{\three}  \omega(b-a))=0 \iff  \omega \in \frac{\pi}{2c_{\three}(b-a)}+\frac{\pi}{c_{\three}(b-a)} \Z.
\end{align*}
With these observations, we find $\omega\in \R$ with  $ \Psi_{i\omega,{\one}}^{12}(b)= \Psi_{i\omega,{\two}}^{11}(b) =0$ if and only if there are integers $k,\ell \in \Z$ with
\[
 \frac{\pi}{c_{\one}(b-a)} k = \frac{\pi}{2c_{\two}(b-a)}+\frac{\pi}{c_{\two}(b-a)}\ell, \text{ that is, } \frac{c_{\one}}{c_\two}=\frac{2k}{1+2\ell}.
\]
Similarly, we find $\omega\in \R$ with $\Psi_{i\omega,{\one}}^{12}(b)= \Psi_{i\omega,{\three}}^{11}(b) =0$ if and only if there are integers $k,\ell\in \Z$ with
\[
   \frac{c_{\one}}{c_\three}=\frac{2k}{1+2\ell}.
\]
Finally, there exists $\omega\in \R$ with $\Psi_{i\omega,{\two}}^{11}(b)=\Psi_{i\omega,{\three}}^{11}(b)=0$ if and only if there are integers $k,\ell\in \Z$ such that 
\[
   \frac{\pi}{2c_{\two}(b-a)}+\frac{\pi}{c_{\two}(b-a)}k = \frac{\pi}{2c_{\three}(b-a)}+\frac{\pi}{c_{\three}(b-a)}\ell,\text{ that is, } \frac{c_{\two}}{c_{\three}}=\frac{1+2k}{1+2\ell}.
\]

Finally, assume that either (hence both) statements (i), and (ii), are true. We carry out the argument only for the case $ \frac{c_{\one}}{c_{\two}} \in \mathbb{E}/\mathbb{O}$ as the other cases can be dealt with similarly. Note that the above proof shows that any pair $(k,\ell)\in \N\times \N$ such that 
\begin{equation}\label{eq:kl}
  \frac{c_{\one}}{c_{\two}}= \frac{2k}{1+2\ell}
\end{equation}
yields an $\omega$ with the desired properties. Note that two distinct pairs lead to different $\omega$'s. Thus, if $(k,\ell)$ satisfies \eqref{eq:kl}, then for every odd $q$, we get that $(kq,\frac{q-1}{2}+\ell q)$ yields
\[
 \frac{2kq}{1+2(\frac{q-1}{2}+\ell q)}=\frac{2kq}{(1+2\ell)q} =\frac{2k}{1+2\ell}= \frac{c_{\one}}{c_{\two}},
\]which eventually shows the assertion.
\end{proof}
Finally, we can summarise our characterisation of semi-uniform stability/asymptotic stability for the case of constant coefficients.
\begin{thm}\label{thm:charexamconst} Let $(A,\dom(A))$ as well as $(T(t))_{t\geq 0}$ be given by Theorem \ref{thm:genexampl}. Assume that $\mathcal{H}$ is constant, that is, given by \eqref{eq:Hconst} and let $\Psi$ be given by \eqref{eq:psik}. Then the following conditions are equivalent:
\begin{enumerate}[(i)]
  \item $(T(t))_{t\geq 0}$ is not semi-uniformly stable.
  \item There exists $\omega\in \R$ with
\[
    |\Psi_{i\omega,{\one}}^{12}(b)\Psi_{i\omega,\two}^{11}(b)|+|\Psi_{i\omega,{\one}}^{12}(b)\Psi_{i\omega,\three}^{11}(b)|+|\Psi_{i\omega,{\two}}^{11}(b)\Psi_{i\omega,\three}^{11}(b)| =0.
    \]
\item At least one of the following statements is true
\[
    \frac{c_{\one}}{c_{\two}} \in \mathbb{E}/\mathbb{O},\quad     \frac{c_{\one}}{c_{\three}} \in \mathbb{E}/\mathbb{O},\quad     \frac{c_{\two}}{c_{\three}} \in \mathbb{O}/\mathbb{O},
\]where $c_k=\sqrt{\rho_k/T_k}$, $k\in \{\one,\two,\three\}$.
\item $A$ has infinitely many eigenvalues on the imaginary axis.
\end{enumerate}
\end{thm}
\begin{proof} The equivalence of (i)~and (ii)~follows from Theorem \ref{thm:stabgen}. 

The equivalence of (ii)~and (iii)~follows from Proposition \ref{prop:instabpart}. 

The condition in (iv)~implies (i)~by Theorem \ref{thm:main1}. Consequently, by what we have already shown, (iv)~implies (ii), which, by Proposition \ref{prop:instabpart}, leads to the existence of infinitely many $\omega\in \R$ satisfying (ii). By the concluding statement in Theorem \ref{thm:stabgen} it follows that $A$ has infinitely many eigenvalues on the imaginary axis; which concludes the proof.
\end{proof}
As an immediate corollary we obtain the following:
\begin{cor}\label{cor:charsus} In the situation of Theorem \ref{thm:charexamconst} the following conditions are equivalent:
\begin{enumerate}
  \item[(i)]  $(T(t))_{t\geq 0}$ is semi-uniformly stable.
  \item[(ii)]  $(T(t))_{t\geq 0}$ is asymptotically stable.
  \item[(iii)]  $ \frac{c_{\one}}{c_{\two}},   \frac{c_{\one}}{c_{\three}}  \notin \mathbb{E}/\mathbb{O}$ and $ \frac{c_{\two}}{c_{\three}} \notin \mathbb{O}/\mathbb{O}$.
\end{enumerate}
\end{cor}
\begin{proof}
Conditions (i)~and (ii)~are equivalent by Theorem \ref{thm:main1}; the remaining equivalence follows from Theorem \ref{thm:charexamconst}.
\end{proof}

\section{Vibrating strings and exponential stability}\label{Sec:4}

We complement the characterisation of exponential stability for our system of vibrating strings in the following. For this, we use the reformulation of exponential stability from Section \ref{sec:exprev}. Recall from \cite[Section 6]{TW22_expstab} that condition $\sup_{\omega\in \R}\|\Psi_{i\omega} \|<\infty$ is satisfied if both $\rho_k$ and $T_k$ are of bounded variation for all $k\in \{\one,\two,\three\}$. This is particularly the case, if $\mathcal{H}$ is constant. 


We start off with a result similar to that of Theorem \ref{thm:mainex}. For this we denote for $v_a=\big(\begin{bmatrix} \nu_{a,k} \\ F_k(a) \end{bmatrix}\big)_{k\in \{\one,\two,\three\}}\in \CC^6$ and $\omega\in \R$, $v_{v_a}^{(\omega)}\colon [a,b]\to \CC^6 = (\CC^2)^3$ to be a solution of \eqref{eq:exampleode} with $\lambda=i\omega$ and $v_{v_a}^{(\omega)}(a)=v_a$. 
\begin{prop}\label{prop:mainexexst} 
Let $(A,\dom(A))$ as well as $(T(t))_{t\geq 0}$ be given by Theorem \ref{thm:genexampl}. Furthermore, assume that $\sup_{\omega\in \R}\|\Psi_{i\omega} \|<\infty$, where $\Psi_{i\omega}$ is given by (\ref{eq:exampleode}). 
Furthermore, let $(\omega_n)_{n\in {\mathbb N}}$ be a sequence in $\R$. Then the following conditions are equivalent:
\begin{enumerate}[(i)]
\item The set
\begin{multline*}
  \{ v_a \in {\mathbb C}^{6} \mid{ \liminf_{n\to\infty}\Big(\|W_B\begin{bmatrix} \Psi_{i\omega_n} (b) \\ I_{6}\end{bmatrix}}v_a\| +\\ | v_a^* \left[\begin{smallmatrix} \Psi_{i\omega_n} (b)^* & I_{6} \end{smallmatrix}\right]  \left[\begin{smallmatrix} P_1 & 0 \\ 0 & -P_1 \end{smallmatrix} \right] \left[\begin{smallmatrix} \Psi_{i\omega_n} (b) \\ I_{6} \end{smallmatrix}\right] v_a |\Big) =0\}
\end{multline*}
contains only the zero element. 
\item 
  Let $v_a\in \CC^6$ and $v^{(n)}\coloneqq v_{v_a}^{(\omega_n)}$;  $v^{(n)}(b)\coloneqq \big(\begin{bmatrix} \nu_{b,k}^{(n)} \\ F_k^{(n)}(b) \end{bmatrix}\big)_{k\in \{\one,\two,\three\}}$. If
\begin{multline*}
\nu_{a,\one} = F_{\two}(a)=F_{\three}(a)=0\text{ and } \\
  \liminf_{n\to\infty}\Big( |\nu_{b,\one}^{(n)}-\nu_{b,\two}^{(n)}|+|\nu_{b,\one}^{(n)}-\nu_{b,\three}^{(n)}|+|\beta\nu_{b,\one}^{(n)} +F_{\one}^{(n)}(b)+F_{\two}^{(n)}(b)+F_{\three}^{(n)}(b)| \\
 + |v^{(n)}(b)^*P_1v^{(n)}(b) - v_a^*P_1v_a|\Big) =0, 
\end{multline*}
then $v_a=0$.
\item Let $v_a\in \CC^6$ and $v^{(n)}\coloneqq v_{v_a}^{(\omega_n)}$; $v^{(n)}(b)\coloneqq \big(\begin{bmatrix} \nu_{a,k}^{(b)} \\ F_k^{(n)}(b) \end{bmatrix}\big)_{k\in \{\one,\two,\three\}}$. If
\begin{align*}
 &  \nu_{a,\one}=F_{\two}(a)=F_{\three}(a) = 0 \quad \text{and}\\
 & \liminf_{n\to\infty} \Big(|\nu_{b,\one}^{(n)}|+|\nu_{b,\two}^{(n)}|+|\nu_{b,\three}^{(n)}|+|F_{\one}^{(n)}(b)+F_{\two}^{(n)}(b)+F_{\three}^{(n)}(b)|\Big) =0, 
\end{align*}
then $v_a=0$.
\end{enumerate}
\end{prop}
\begin{proof}
Recall that the solution of \eqref{eq:exampleode} for $\lambda=i\omega_n$ is represented by $v^{(n)}(\cdot)= \Psi_{i\omega_n}(\cdot)v_a$. In this case,  $v^{(n)}(b)= \Psi_{i\omega_n}(b)v_a$ for all $n\in \N$. 

(i)$\Leftrightarrow$(ii).
We reformulate the expression in the $\liminf$ in the set given in (i). Let $v_a$ belong to this set. By passing to a subsequence, which is not relabelled we may assume that the $\liminf$ in (ii) can be replaced by a limit. Since we considered non-negative terms in the $\liminf$-expression only, we obtain that the individual terms are null-sequences. Thus, the terms in the norms read
\begin{align}
\label{eq:proofeqexp2}
   &\lim_{n\to\infty}W_B \begin{bmatrix} v^{(n)}(b) \\ v_a \end{bmatrix} =0\text{ and } \\
    &\lim_{n\to\infty} v^{(n)}(b)^*P_1v^{(n)}(b) - v_a^*P_1v_a=0,
\end{align} 
respectively.

Next, since $v_a=\big(\begin{bmatrix} \nu_{a,k} \\ F_k(a) \end{bmatrix}\big)_{k\in \{\one,\two,\three\}}$ and $v^{(n)}(b)\coloneqq \big(\begin{bmatrix} \nu_{b,k}^{(n)} \\ F_k^{(n)}(b) \end{bmatrix}\big)_{k\in \{\one,\two,\three\}}$, the conditions encoded in $\lim_{n\to\infty}\|W_B \left[\begin{smallmatrix} v^{(n)}(b) \\ v_a \end{smallmatrix} \right]\|= 0$ read
\begin{align*}
 &\nu_{a,I} = 0,\quad F_{\two}(a)=F_{\three}(a) = 0 \\
 \lim_{n\to\infty} &\ \nu_{b,\one}^{(n)} -\nu_{b,\two}^{(n)}=\lim_{n\to\infty}\nu_{b,\one}^{(n)} -\nu_{b,\three}^{(n)} =0 \\
  \lim_{n\to\infty}&\  \beta \nu_{b,\one}^{(n)}+ F_{\one}^{(n)}(b)+F_{\two}^{(n)}(b)+F_{\three}^{(n)}(b) =0.
 \end{align*}
 
 Similarly, we may reformulate the conditions in (ii). Indeed, by passing to subsequence, where one write limits instead of $\liminf$, we obtain the limit expression in (i) for this subsequence. 
 
 (iii)$\Leftrightarrow$(ii) We begin by deriving some equalities. For $n\in \N$
  \begin{align*}
v^{(n)}(b)^*P_1&v^{(n)}(b) - v_a^*P_1v_a \\ =&\ 2\sum_{k\in \{\one,\two,\three\}} \big( \begin{bmatrix} \nu_{b,k}^{(n)} \\ F_k^{(n)}(b) \end{bmatrix}^*\begin{bmatrix} 0 & 1\\ 1 & 0\end{bmatrix}\begin{bmatrix} \nu_{b,k}^{(n)} \\ F_k^{(n)}(b) \end{bmatrix}-\begin{bmatrix} \nu_{a,k} \\ F_k(a) \end{bmatrix}^*\begin{bmatrix} 0 & 1\\ 1 & 0\end{bmatrix}\begin{bmatrix} \nu_{a,k} \\ F_k(a) \end{bmatrix}  \big) \\
 =&\  2\sum_{k\in \{\one,\two,\three\}} \Re \left[ (\nu_{b,k}^{(n)})^*F_{k}^{(n)}(b)\right] - \Re[\nu_{a,k}^*F_{k}(a)] \\
 =&\ 2 \sum_{k\in \{\one,\two,\three\}} \Re \left[ (\nu_{b,k}^{(n)})^*F_{k}^{(n)}(b)\right] ,
 \end{align*}
 where we have used the (common) conditions on $v_a$. Hence
 \begin{equation}
 \label{eq:36HZ}
   v^{(n)}(b)^*P_1v^{(n)}(b) - v_a^*P_1v_a = v^{(n)}(b)^*P_1v^{(n)}(b) = 2\sum_{k\in \{\one,\two,\three\}} \Re \left[ (\nu_{b,k}^{(n)})^*F_{k}^{(n)}(b)\right] .
 \end{equation}
Furthermore, since $v^{(n)}(b)= \Psi_{i\omega_n}(b)v_a$ and $\Psi_{i\omega}$ is assumed to be bounded, we have that the sequences $(\nu_{b,k}^{(n)}, F_{k}^{(n)}(b))_{n \in {\mathbb N}}$, $k\in \{\one,\two,\three\}$, are bounded.

Next we prove that  (iii)$\Rightarrow$(ii). By possibly passing to a subsequence, we may assume that the $\liminf$ can be replaced by $\lim$. 

From (\ref{eq:36HZ}) we deduce that
\[
   |v^{(n)}(b)^*P_1v^{(n)}(b) - v_a^*P_1v_a|  \leq 2 \sum_{k\in \{\one,\two,\three\}} |(\nu_{b,k}^{(n)})||F_{k}^{(n)}(b)|.
\]
Combining this with the property that  $(F_{k}^{(n)}(b))_{n\in \N}$ is a bounded sequence for all $k\in  \{\one,\two,\three\}$, and that $\lim_{n \rightarrow \infty}\nu_{b,k}^{(n)}=0$, we find that  $\lim_{n \rightarrow \infty}    |v^{(n)}(b)^*P_1v^{(n)}(b) - v_a^*P_1v_a| =0$.

 The inequalities
\begin{align*}
|\nu_{b,\one}^{(n)}-\nu_{b,\two}^{(n)}|+|\nu_{b,\one}^{(n)}-\nu_{b,\three}^{(n)}| & \leq 2 |\nu_{b,\one}^{(n)}|+|\nu_{b,\two}^{(n)}|+|\nu_{b,\three}^{(n)}|\text{ and } \\
|\beta\nu_{b,\one}^{(n)} +F_{\one}^{(n)}(b)+F_{\two}^{(n)}(b)+F_{\three}^{(n)}(b)| & \leq |\beta\nu_{b,\one}^{(n)}| +|F_{\one}^{(n)}(b)+F_{\two}^{(n)}(b)+F_{\three}^{(n)}(b)|
\end{align*}
 yield the remaining convergences.  Thus, (iii) is sufficient for (ii).

 Next we show that (ii)$\Rightarrow$(iii). Again, we may assume, by possibly passing to a subsequence, not restricting the generality, that $\liminf$ can be replaced by $\lim$. Next, starting with the term on the right-hand side of (\ref{eq:36HZ}), we compute for $n\in \N$
 \begin{align}
 \nonumber
  \sum_{k\in \{\one,\two,\three\}} (\nu_{b,k}^{(n)})^*F_{k}^{(n)}(b) =&\  (\nu_{b,\one}^{(n)})^*\left[ F_{\one}^{(n)}(b) + F_{\two}^{(n)}(b) +F_{\three}^{(n)}(b) \right] +\\
  \nonumber
  &\ [ \nu_{b,\two}^{(n)} - \nu_{b,\one}^{(n)}]^* F_{\two}^{(n)}(b) +  [ \nu_{b,\three}^{(n)} - \nu_{b,\one}^{(n)}]^* F_{\three}^{(n)}(b) \\
 \nonumber
  =&\
  (\nu_{b,\one}^{(n)})^*\left[ \beta \nu_{b,\one}^{(n)} + F_{\one}^{(n)}(b) + F_{\two}^{(n)}(b) +F_{\three}^{(n)}(b) \right] - \beta |\nu_{b,\one}^{(n)}|^2 +\\
  \label{eq:37HZ}
 &\ [ \nu_{b,\two}^{(n)} - \nu_{b,\one}^{(n)}]^* F_{\two}^{(n)}(b) +  [ \nu_{b,\three}^{(n)} - \nu_{b,\one}^{(n)}]^* F_{\three}^{(n)}(b).
\end{align}
Since $(\nu_{b,\one}^{(n)})_{n\in \N}, (F_{\two}^{(n)}(b))_{n\in \N}$ and $(F_{\three}^{(n)}(b))_{n\in \N}$ are bounded sequences, the conditions in item (ii) combined with (\ref{eq:36HZ}) and (\ref{eq:37HZ}) imply that $\lim_{n \rightarrow \infty} \beta |\nu_{b,\one}^{(n)}|^2 =0$. Using that $\beta \neq 0$, this gives that $\lim_{n \rightarrow \infty} |\nu_{b,\one}^{(n)}|^2 =0$. Substituting this property into the equalities in item (ii) gives the equalities in item (iii).
\end{proof}
 
\begin{thm}\label{thm:charexpoex} Let $(A,\dom(A))$ as well as $(T(t))_{t\geq 0}$ be given by Theorem \ref{thm:genexampl}. Furthermore, assume that $\sup_{\omega\in \R}\|\Psi_{i\omega} \|<\infty$, where $\Psi_{i\omega}$ is given by (\ref{eq:exampleode}). Then the following conditions are equivalent:
\begin{enumerate}
\item[(i)] $(T(t))_{t\geq 0}$ is exponentially stable;
\item[(ii)] for all sequences in $(\omega_k)_{k \in {\mathbb N}}$ in $\R$
\[
   \liminf_{k\to\infty} \Big(|\Psi_{i\omega_k,{\one}}^{12}(b)\Psi_{i\omega_k,\two}^{11}(b)|+|\Psi_{i\omega_k,{\one}}^{12}(b)\Psi_{i\omega_k,\three}^{11}(b)|+|\Psi_{i\omega_k,{\two}}^{11}(b)\Psi_{i\omega_k,\three}^{11}(b)|\Big) >0.
\]
\end{enumerate}
\end{thm}
\begin{proof}
 The set up is now similar to the proof of Theorem \ref{thm:stabgen}. We use Theorem \ref{thm:exp2} for the characterisation of exponential stability and Proposition \ref{prop:mainexexst} for the reformulation of (ii) in Theorem \ref{thm:exp2}.
 
(i)$\Rightarrow$(ii) We prove that not (ii) implies not (i). For this, we distinguish two cases. If $(\omega_k)_{k\in \N}$ is a bounded sequence in $\R$ with 
\[
   \liminf_{k\to\infty} \Big(|\Psi_{i\omega_k,{\one}}^{12}(b)\Psi_{i\omega_k,\two}^{11}(b)|+|\Psi_{i\omega_k,{\one}}^{12}(b)\Psi_{i\omega_k,\three}^{11}(b)|+|\Psi_{i\omega_k,{\two}}^{11}(b)\Psi_{i\omega_k,\three}^{11}(b)|\Big) =0.
\]
Then, by the continuity of $\omega\mapsto \Psi_{i\omega}$,
\[
   \Big(|\Psi_{i\omega,{\one}}^{12}(b)\Psi_{i\omega,\two}^{11}(b)|+|\Psi_{i\omega,{\one}}^{12}(b)\Psi_{i\omega,\three}^{11}(b)|+|\Psi_{i\omega_k,{\two}}^{11}(b)\Psi_{i\omega,\three}^{11}(b)|\Big) =0
\]for one accumulation value $\omega$ of $(\omega_k)_{k\in \N}$. Hence, by  Theorem \ref{thm:stabgen}, the semi-group is not semi-uniformly stable, and, thus, not exponentially stable. 

It remains to consider the case $(\omega_k)_{k\in \N}$ is an unbounded sequence in $\R$ with no bounded subsequence such that
\[
   \liminf_{k\to\infty} \Big(|\Psi_{i\omega_k,{\one}}^{12}(b)\Psi_{i\omega_k,\two}^{11}(b)|+|\Psi_{i\omega_k,{\one}}^{12}(b)\Psi_{i\omega_k,\three}^{11}(b)|+|\Psi_{i\omega_k,{\two}}^{11}(b)\Psi_{i\omega_k,\three}^{11}(b)|\Big) =0.
\]
Without restriction that we may assume $\omega_k\to \infty$. Moreover, using the boundedness of $\omega\mapsto \Psi_{i\omega}$, we may choose a subsequence (not relabelled) such that $(\Psi_{i\omega_k}(b))_{k\in \N}$ converges as $k\to\infty$; we let
\[
    P_{l}^{ms}\coloneqq \lim_{k\to\infty} \Psi^{ms}_{i\omega_k,l}(b)\quad(l\in \{\one,\two,\three\}, m,s\in \{1,2\})
\]
The assumption implies that two of the three numbers $P_{{\one}}^{12}, P_{\two}^{11}, P_{\three}^{11}$ are zero.

Consider the case $P_{{\one}}^{12}=P_{\two}^{11}=0$. By \cite[Lemma 3.2]{TW22_expstab}, $\omega\mapsto(x\mapsto  \Psi_{i\omega}(x)^{-1})$ is uniformly bounded. Thus, $P_{\one}^{22}\neq 0$, for otherwise $\|\Psi_{i\omega}(b)^{-1}\|\to\infty$. Next, we define $\nu_{a,\one}=\nu_{a,\three}=0$ as well as $F_{\two}(a)=F_{\three}(a)=0$. Moreover, $\nu_{a,\two}=1$ and $F_{\one}(a)= (P_{\one}^{22})^{-1}(-P_{\two}^{21})$. Then, $v_a\neq 0$ and, hence, $v^{(n)}\coloneqq  \Psi^{ms}_{i\omega_n}(\cdot)v_a\neq 0$. Then, $\nu_{b,\three}^{(n)}=F_{\three}^{(n)}(b)=0$ since $v^{(n)}_{\three}=0$. 

Next,
\begin{align*}
   \nu_{b,\one}^{(n)}&= \Psi_{i\omega_n,\one}^{11}(b)\nu_{a,\one} + \Psi_{i\omega_n,\one}^{12}(b)F_\one(a)\to 0\quad (n\to\infty) \\
   \nu_{b,\two}^{(n)}& = \Psi_{i\omega_n,\two}^{11}(b)\nu_{a,\two} + \Psi_{i\omega_n,\two}^{12}(b)F_\two(a)\to 0 \quad(n\to\infty).
\end{align*}
Next, we show $F_{\one}^{(n)}(b)+F_{\two}^{(n)}(b)+F_{\three}^{(n)}(b)\to 0$ as $n\to\infty$. For this, it suffices to consider
\begin{align*}
   F_{\one}^{(n)}(b)+F_{\two}^{(n)}(b) & = \Psi_{i\omega_n,\one}^{21}(b)\nu_{a,\one} + \Psi_{i\omega_n,\one}^{22}(b)F_\one(a)+\Psi_{i\omega_n,\two}^{21}(b)\nu_{a,\two} + \Psi_{i\omega_n,\two}^{22}(b)F_\two(a) \\ &= 0 + \Psi_{i\omega_n,\one}^{22}(b)(P_{\one}^{22})^{-1}(-P_{\two}^{21})) +\Psi_{i\omega_n,\two}^{21}(b) + 0\to 0\quad(n\to\infty)
\end{align*}
which establishes all conditions in (iii) of Proposition \ref{prop:mainexexst} albeit $v_a\neq 0$. Hence, Theorem \ref{thm:exp2} shows that $(T(t))_{t\geq 0}$ is not exponentially stable. 

The remaining two cases are either a variant of the one just developed or can be dealt with analogously to the second case considered in Theorem \ref{thm:stabgen} subject to the modifications along the lines of the case just carried out here.

(ii)$\Rightarrow$(i) The statement can be shown along the same lines as Proposition \ref{prop:omeganonzero} in order to establishing condition (iii) in Proposition \ref{prop:mainexexst}. Thus, $(T(t))_{t\geq0}$ is exponentially stable by Theorem \ref{thm:exp2} 
\end{proof}
\begin{rem}
Note that in the case of constant $\mathcal{H}$ it suffices to consider sequences tending to $\infty$ (or $-\infty$) in (ii) in Theorem \ref{thm:charexpoex}, by Proposition \ref{prop:instabpart}.
\end{rem}

In the remaining part of this section, we restrict ourselves to constant $\mathcal{H}$; that is, we use \eqref{eq:Hconst} and recall $c_{\one}, c_{\two}, c_{\three}$ from \eqref{eq:ck}. We emphasise that in this case, $\omega\mapsto \Psi_{i\omega}$ is bounded, as $\mathcal{H}$ is of bounded variation, see \cite{TW22_expstab}. The next example contains a set up yielding exponential stability.
\begin{example} 
\label{E:5.4}
Let $c_{\one}=c_{\two}=1$ and $c_{\three}=2$, $a=0$, $b=1$. Then the corresponding set up of vibrating strings is exponentially stable. Appealing to \eqref{eq:psi1}, \eqref{eq:psi2}, and \eqref{eq:psi3} together with Theorem \ref{thm:charexpoex}, we may only consider for $\omega\in \R$ the expression
\[
\eta(\omega)\coloneqq |\sin(\omega)\cos(\omega)|^2 +|\sin(\omega)\cos(2\omega)|^2+|\cos(\omega)\cos(2\omega)|^2 
\]
and analyse whether this is uniformly bounded away from zero as $\omega$ ranges over $\R$. 

Using elementary formulas for $\sin$ and $\cos$ and putting $s\coloneqq \sin(\omega)$ and $c\coloneqq \cos(\omega)$, we deduce for all $\omega$
\[
  \eta(\omega) = s^2c^2 + (c^2-s^2)^2 = s^2(1-s^2) + (1-2s^2)^2 =: f(s).
\]
which as a sum of non-negative terms can only be zero, if either terms vanish. It is not hard to show that the minima of the polynomial $f$ are for $s^2=1/2$. So $f(s) \geq 1/4$ for all $s^2\in [0,1]$
and thus 
\[
  \inf_{\omega\in \R} \eta(\omega)\geq \frac{1}{4}.
\]
Using Theorem \ref{thm:charexpoex} we conclude that this system of vibrating strings yields an exponentially stable semigroup. \hfill $\Box$
\end{example}

We conclude this section by providing a setting, where the semigroup corresponding to the port-Hamiltonian system of Figure \ref{fig:1} is semi-uniformly stable but not exponentially stable. The core is Kronecker's theorem from number theory (see, e.g., \cite[Theorem 2.21]{Haaseetal}). For this we define $\T\coloneqq \{z\in \CC\mid{ | z|=1}\}$. For $m\in \N$ and $a=(a_1,\ldots,a_m)\in \T^m$, we define
\[
   \phi_a \colon \T^m \to \T^m, (z_1,\ldots,z_m)
   \mapsto (a_1z_1,\ldots,a_m z_m).
\]
\begin{thm}[Kronecker]
\label{thm:kronecker} 
Let $ (a_1,\ldots,a_m )\in \T^m$. Then the following conditions are equivalent
\begin{enumerate}[(i)]
  \item for all $z\in \T^m$ the set $\{\phi_a^k(z)\mid{k\in \N}\}$ is dense in $\T^m$.
  \item for all $k_1,\ldots, k_m\in \Z$ with $a_1^{k_1}a_2^{k_2}\cdots a_m^{k_m} =1$, we have $k_1=k_2=\ldots=k_m=0$.
\end{enumerate}
\end{thm}
\begin{proof}
The proof is a (straightforward) combination of \cite[Example 2.18 and Theorem 2.21]{Haaseetal}.
\end{proof}

We turn back to our example of vibrating strings and assume that $\mathcal{H}$ is constant, that is, given by \eqref{eq:Hconst} and recall $c_k=\sqrt{\rho_k/T_k}$, $k\in \{\one,\two,\three\}$. In this case, we may recall (\ref{eq:psi1})--(\ref{eq:psi3}).

We define for $x,y\in\R$, the equivalence relation $x\sim y$ via $x-y\in \Z$ and $\tilde{\T}\coloneqq \R/_\sim$. We say that $x,y\in \R$ are \textit{linearly dependent in the $\Z$-module $\tilde{\T}$}, if there exists integers $k,n\in \Z$ such that $kx\sim ny$. In case they are not linearly dependent, they are called \textit{linearly independent in the $\Z$-module $\tilde{\T}$}.
\begin{thm} 
Let $(A,\dom(A))$ as well as $(T(t))_{t\geq 0}$ be given by Theorem \ref{thm:genexampl} with $\mathcal{H}$ is constant.

Assume that all of the pairs $(c_{\one},c_{\two})$, $(c_{\one},c_{\three})$, and $(c_{\two},c_{\three})$ are linearly independent in the $\Z$-module $\tilde{\T}$.
Then $(T(t))_{t\geq 0}$ is semi-uniformly stable, but not exponentially stable.
\end{thm}
\begin{proof}
As all the pairs are independent in the $\Z$-module, the condition (iii) in Corollary \ref{cor:charsus} is satisfied yielding that $(T(t))_{t\geq 0}$ is semi-uniformly stable. 

By Theorem \ref{thm:kronecker} the set 
\[
   \{ (e^{i c_{\one}2\pi n},e^{i c_{\two}2\pi n}); n\in \N\}  =     \{ \phi^n_{(e^{i c_{\one}2\pi },e^{i c_{\two}2\pi)})}(1,1); n\in \N\}\subseteq \T^2
\]
is dense ${\mathbb T}^2$. In particular, $(1,i)$ can be approximated arbitrarily well, which yields that
\[
\{ (\sin(c_{\one}2\pi n),\cos(c_{\two}2\pi n)); n\in \N\}
\]
accumulates at $0$. Hence, $\omega_n\coloneqq 2\pi n/(b-a)$ defines a sequence such that 
\[
   \liminf_{n\to\infty} \Big(|\Psi_{i\omega_n,{\one}}^{12}(b)\Psi_{i\omega_n,\two}^{11}(b)|+|\Psi_{i\omega_n,{\one}}^{12}(b)\Psi_{i\omega_n,\three}^{11}(b)|+|\Psi_{i\omega_n,{\two}}^{11}(b)\Psi_{i\omega_n,\three}^{11}(b)|\Big) =0,
\]
which implies that the semigroup is not exponentially stable by Theorem \ref{thm:charexpoex}.
\end{proof}

\section{Conclusion}

For port-Hamiltonian systems on a one-dimensional spatial domain we have presented equivalent characterisations of asymptotic stability. One of the characterisations gives that the system is asymptotically stable when there are no eigenvalues on the imaginary axis. However, in general these eigenvalues are not easy to calculate, and so there is still a need for (easy) sufficient conditions, like there is for exponential stability, see \cite{Villegas_2009}. Furthermore, our example shows that stability is very sensitive with respect to its (constant) parameters, and thus asymptotic stability lacks robustness. We supported this observation also by looking at exponential stability. A logical follow-up question would be if this is also the case for spatial varying parameters. 

There is a strong link between observability/controllability of systems and the stabilisation of it, see e.g.\ \cite{CuWe19} and the references therein. For wave equations on a network, like we have in Section \ref{Sec:3}, observability and controllability has been well studied in the book by D\'{a}ger and Zuazua, \cite{DaZu06}. Hence  future research should study the relation between their and our results.

\appendix
\section{Well-posedness of Ordinary Differential Equations with measurable coefficients}\label{app:wp}

We provide the well-posedness theorem underlying the unique existence of the fundamental solutions considered  in the main body of the manuscript. The strategy to show well-posedness uses exponentially weighted $L_2$-spaces. In an o.d.e.-context, this has been employed for instance in \cite[Chapter 4]{STW_EE21} and the references therein. For convenience, we write out the argument in this special situation.

\begin{thm}\label{thm:wpode} Let $\mathcal{K}\in L_\infty(a,b)^{d \times d}$. Then for all $u_0\in \CC^d$ there exists a unique $u\in H^1(a,b)^d$ such that
\[
   u' = \mathcal{K}u,\quad u(a+)=u_0.
\]
\end{thm}
\begin{proof} Without loss of generality, $(a,b)=(0,1)$. 
 We start off with existence of the solution. For this, let $\rho>0$ and consider the Hilbert space
 \[
    L_{2,\rho}(0,1)^d\coloneqq (L_2(0,1)^d; (v,w)\mapsto \int_0^1 \langle v(x),w(x)\rangle_{{\mathbb C}^d} e^{-2\rho x} \d x).
 \]
 Define
 \[
    \Psi\colon L_{2,\rho}(0,1)^d\to L_{2,\rho}(0,1)^d
 \]
 by
 \[
    (\Psi u)(x) \coloneqq u_0 +\int_0^x \mathcal{K}(s)u(s)\d s.
 \]
 Then, clearly $\Psi$ is well-defined. Moreover, for $u,v\in L_{2,\rho}(0,1)^d$ we obtain
 \begin{align*}
   \|\Psi u-\Psi v\|^2_{L_{2,\rho}} & = \int_0^1 \|\int_0^x \mathcal{K}(s)u(s)-\mathcal{K}(s)v(s)\d s\|^2 \e^{-2\rho x}\d x \\
   &\leq \int_0^1(\int_0^x \|\mathcal{K}(s)\|\| u(s)-v(s)\| \d s)^2\e^{-2\rho x} \d x \\
   &\leq \|\mathcal{K}\|_\infty^2 \int_0^1 (\int_0^x \| u(s)-v(s)\|\e^{-\rho s}\e^{\rho s} \d s)^2 \e^{-2\rho x}\d x \\
   & \leq \|\mathcal{K}\|_\infty^2 \int_0^1 \int_0^x \| u(s)-v(s)\|^2\e^{-2\rho s} \d s \int_0^x \e^{2\rho s}\d s \e^{-2\rho x}\d x \\ 
   & \leq  \|\mathcal{K}\|_\infty^2 \| u-v\|_{L_{2,\rho}}^2 \int_0^1 \int_0^x \e^{2\rho s}\d s \e^{-2\rho x}\d x \\
   & \leq \|\mathcal{K}\|_\infty^2 \| u-v\|_{L_{2,\rho}}^2  \frac{1}{2 \rho}.
 \end{align*}
 Choosing $2\rho> \|\mathcal{K}\|_\infty^2$, we get that $\Psi$ is a strict contraction and assumes a unique fixed point $u$. In consequence,
 \[
   u(x) = \Psi (u)(x) = u_0 +\int_0^x \mathcal{K}(s)u(s)\d s,
 \]which by Fubini's theorem leads to $u\in H^1(0,1)^d$ with 
 \[
    u'(x)=  \mathcal{K}(x)u(x), \quad \text{a.e. } x\in (a,b).
 \] 
 Moreover, Lebesgue's dominated convergence theorem confirms $u(0+)=u_0$; hence, $u$ is a solution of the desired equation.  
 
 To address uniqueness, it suffices to show that the only solution of the differential equation with $u_0=0$ is zero. For this, we integrate the desired o.d.e.~and obtain
 \[
    u(x) = \int_0^x \mathcal{K}(s)u(s)\d s.
 \]
 Hence, 
 \[
    \|u(x)\|\leq \|\mathcal{K}\|_\infty \int_0^x \|u(s)\| \d s
 \]and Gronwall's inequality leads to
 \[
   \|u(t)\|\leq 0; \text{that is, } u=0.\qedhere
 \]
\end{proof}

\begin{cor}\label{cor:exfs} Under the assumptions of Theorem \ref{thm:wpode}, there exists a uniquely determined continuous function $\Gamma \colon [a,b]\to \CC^{d\times d}$ with $\Gamma (\cdot) u_0$ being the unique solution of
\[
     u'(x)=\mathcal{K}(x)u(x),\quad u(a)=u_0.
\]It follows that $\Gamma(0)=I_d$.
\end{cor}
\begin{proof}
  For uniqueness, let $\Gamma_1$ share the properties of $\Gamma$. Then for $u_0\in \CC^d$, it follows that $\Gamma_1(\cdot)u_0 =\Gamma(\cdot)u_0$ by the uniqueness result in Theorem \ref{thm:wpode}. Thus, $\Gamma_1=\Gamma$. To address existence, we note that the solution operator $S\colon \CC^d\to H^1(a,b)^d$ assigning to each initital value $u_0$ the corresponding solution constructed in Theorem \ref{thm:wpode} is linear. By the Sobolev embedding theorem, $A_x\colon u_0\mapsto (Su_0)(x)$ is well-defined; as it is linear, we infer that for all $x\in [a,b]$ there is a matrix $\Gamma(x)$, which represents $A_x$. For the continuity of $\Gamma$ it is enough to see that for all $u_0\in \CC^d$, we have $\Gamma(x)u_0 = A_x u_0 = (S u_0)(x)$ and so $\Gamma(\cdot)u_0 \in H^1(a,b)^d\subseteq C[a,b]^d$, by the Sobolev embedding theorem. The last condition follows from $\Gamma(a)u_0=u(a)=u_0$.
\end{proof}

Finally, we present a continuous dependence result in the coefficients, which is a consequence of Gronwall's inequality -- it was already observed in an earlier version of \cite{TW22_expstab}.

\begin{prop}\label{prop:cd} The mapping
\[
    L_\infty(a,b)^{d\times d} \ni \mathcal{K}\mapsto \Gamma_{\mathcal{K}}\in C([a,b];\CC^{d\times d}),
\]
where $\Gamma_{\mathcal{K}}$ is the fundamental solution of 
\[
     u'(x)=\mathcal{K}(x)u(x),\text{ with }\Gamma_{\mathcal{K}}(0)=I_d,
\]is continuous.
\end{prop}
\begin{proof}
  Let $\mathcal{K},\mathcal{L}\in L_\infty(a,b)^{d\times d}$. For $u_0\in \CC^d$ let $u\coloneqq \Gamma_{\mathcal{K}}(\cdot)u_0$ and $v\coloneqq \Gamma_{\mathcal{L}}(\cdot)u_0$. Then, we compute for (a.e.) $t\in (a,b)$
  \begin{align*}
    \partial_t \| u(t)-v(t)\|^2  &= \Re \langle u'(t)-v'(t),u(t)-v(t)\rangle \\
    & = \Re \langle (\mathcal{K}u)(t)-(\mathcal{L}v)(t),u(t)-v(t)\rangle \\
    & \leq\Big( \|\mathcal{K}-\mathcal{L}\|_\infty \|u(t)\| + \|\mathcal{L}\|_\infty \|u(t)-v(t)\|\Big)\|u(t)-v(t)\|.
  \end{align*}
  Integration over $(a,s)$ yields
  \[
     \| u(s)-v(s)\|^2 \leq  \|\mathcal{K}-\mathcal{L}\|_\infty \int_a^s \|u(t)\|\|u(t)-v(t)\| dt + \|\mathcal{L}\|_\infty \int_a^s \|u(t)-v(t)\|^2 dt.
  \]
  Thus, by Gronwall's inequality,
  \[
   \| u(s)-v(s)\|^2 \leq  \|\mathcal{K}-\mathcal{L}\|_\infty \int_a^s \|u(t)\|\|u(t)-v(t)\| dt \cdot e^{\|\mathcal{L}\|_\infty (b-a)};
 \]
 Taking the supremum over $s$ implies
 \[
    \| u-v\|_\infty \leq \|\mathcal{K}-\mathcal{L}\|_\infty \int_a^b \|u(t)\|dt \cdot e^{\|\mathcal{L}\|_\infty (b-a)} .
 \]
So when $\mathcal{L}\to \mathcal{K}$ in $L_\infty(a,b)^{d\times d}$, then $\| u-v\|_\infty = \|(\Gamma_{\mathcal K} - \Gamma_{\mathcal L})u_0\| \to 0$, which shows that the fundamental solution, $\Gamma_{\mathcal{K}}$, is continuous in $C([a,b];\CC^{d\times d})$.
\qedhere
\end{proof}


\begin{thebibliography}{10}


%

\bibitem{Batty_Duyckaerts_2008}
C.~J.~K. {Batty} and T.~{Duyckaerts}.
\newblock {Non-uniform stability for bounded semi-groups on Banach spaces}.
\newblock {\em {J. Evol. Equ.}}, 8(4):765--780, 2008.

\bibitem{CST20}
R.~Chill, D.~Seifert, and Y.~Tomilov.
\newblock{Semi-uniform stability of operator semigroups and energy decay of damped waves}
\newblock {\em Phil. Trans. R. Soc. A.} 378:20190614.

\bibitem{CuWe19}
R.F. Curtain, and G. Weiss, 
Strong stabilization of (almost) impedance passive systems by
              static output feedback,
{\em Math. Control Relat. Fields}, 9(4): 643--671, 2019

\bibitem{DaZu06}
R. D\'{a}ger and E.\ Zuazua,
{\em Wave propagation, observation and control in {$1\text{-}d$}
              flexible multi-structures},
   Math\'{e}matiques \& Applications (Berlin) [Mathematics \&
              Applications], 50, Springer-Verlag, Berlin, 2006.



\bibitem{GW16}
F.~Gesztesy and M.~Waurick.
\newblock {\em The Callias Index Formula Revisited}
\newblock Lecture Notes in Mathematics 2157, Springer, Berlin, 2016

\bibitem{Haaseetal}
T.~Eisner, B.~Farkas, M.~Haase, and R.~Nagel
\newblock{\em Operator theoretic aspects of ergodic theory.}
\newblock Graduate Texts in Mathematics 272. Cham: Springer, 2015.

\bibitem{GoZM05}
Y.~{Le Gorrec}, H.~Zwart, and B.~Maschke.
\newblock Dirac structures and boundary control systems associated with
  skew-symmetric differential operators.
\newblock {\em SIAM Journal on Control and Optimization}, 44(5):1864--1892,
  2005.
  
\bibitem{JMZ_2015}
B.~{Jacob}, K.~{Morris}, and H.~{Zwart}.
\newblock {\(C_{0}\)-semigroups for hyperbolic partial differential equations
  on a one-dimensional spatial domain}.
\newblock {\em {J. Evol. Equ.}}, 15(2):493--502, 2015.


\bibitem{JZ}
B.~Jacob and H.J. Zwart.
\newblock {\em Linear port-{H}amiltonian systems on infinite-dimensional
  spaces}, volume 223 of {\em Operator Theory: Advances and Applications}.
\newblock Birkh\"{a}user/Springer Basel AG, Basel, 2012.
\newblock Linear Operators and Linear Systems.



\bibitem{PTWW_22}
R.~{Picard}, S.~{Trostorff}, B.~{Watson}, and M.~{Waurick}.
\newblock A structural observation on port-Hamiltonian systems.
\newblock {\em SIAM J. Control Optim.} 61, No. 2, 511-535, 2023.
%





\bibitem{STW_EE21}
C.~Seifert, S.~Trostorff, and M.~Waurick.
\newblock {\em Evolutionary Equations -- Picard's Theorem for Partial Differential Equations, and Applications}
\newblock {Operator Theory: Advances and Applications 287, 2021}

%
\bibitem{TW22_expstab}
S.~Trostorff and M.~Waurick
\newblock Characterisation for Exponential Stability of port-Hamiltonian Systems
\newblock 2022; arXiv:2201.10367

%



\bibitem{Villegas_2009}
J.A. {Villegas}, H.~{Zwart}, Y.~{Le Gorrec}, and B.~{Maschke}.
\newblock {Exponential stability of a class of boundary control systems}.
\newblock {\em {IEEE Trans. Autom. Control}}, 54(1):142--147, 2009.


\end{thebibliography}
\end{document}